\documentclass[a4paper,reqno,11pt]{amsart}

\usepackage{amsmath, amsfonts, amssymb, amsthm, amscd}
\usepackage{graphicx}
\usepackage{psfrag}
\usepackage{perpage}
\usepackage{url}
\usepackage{color}

\usepackage[utf8]{inputenc}
\usepackage[T1]{fontenc}

\usepackage[a4paper,scale={0.72,0.74},marginratio={1:1},footskip=7mm,headsep=10mm]{geometry}

\usepackage{hyperref}

\numberwithin{equation}{section}

\newtheorem{theorem}{Theorem}[section]
\newtheorem{lemma}[theorem]{Lemma}
\newtheorem{proposition}[theorem]{Proposition}
\newtheorem{corollary}[theorem]{Corollary}
\newtheorem{remark}[theorem]{Remark}

\newtheorem{hypothesis}[theorem]{Hypothesis}

\newcommand{\bbE}{{\ensuremath{\mathbb E}} }

\newcommand{\bbP}{{\ensuremath{\mathbb P}} }

\newcommand{\cE}{{\ensuremath{\mathcal E}} }

\newcommand{\cP}{{\ensuremath{\mathcal P}} }

\newcommand{\ga}{\alpha}

\newcommand{\gd}{\delta}

\newcommand{\gep}{\varepsilon}       

\newcommand{\gs}{\sigma}

\renewcommand{\tilde}{\widetilde}          
\DeclareMathSymbol{\leqslant}{\mathalpha}{AMSa}{"36} 
\DeclareMathSymbol{\geqslant}{\mathalpha}{AMSa}{"3E} 
\DeclareMathSymbol{\eset}{\mathalpha}{AMSb}{"3F}     
\newcommand{\dd}{\text{\rm d}}             

\newcommand{\sumtwo}[2]{\sum_{\substack{#1 \\ #2}}} 

\newcommand{\R}{\mathbb{R}}

\newcommand{\Z}{\mathbb{Z}}
\newcommand{\N}{\mathbb{N}}

\def\bs{\boldsymbol}

\newcommand{\PEfont}{\mathrm}

\DeclareMathOperator{\p}{\ensuremath{\PEfont P}}
\DeclareMathOperator{\e}{\ensuremath{\PEfont E}}

\newcommand\bP{\ensuremath{\bs{\mathrm{P}}}}
\newcommand\bE{\ensuremath{\bs{\mathrm{E}}}}

\newcommand{\ind}{{\sf 1}}

\renewcommand{\epsilon}{\varepsilon}
\renewcommand{\theta}{\vartheta}
\renewcommand{\rho}{\varrho}
\renewcommand{\phi}{\varphi}

\newenvironment{myenumerate}{%
\renewcommand{\theenumi}{\arabic{enumi}}%
\renewcommand{\labelenumi}{{\rm(\theenumi)}}%
\begin{list}{\labelenumi}
	{%
	\setlength{\itemsep}{0.4em}%
	\setlength{\topsep}{0.5em}%
	\setlength\leftmargin{2.45em}%
	\setlength\labelwidth{2.05em}%
	\setlength{\labelsep}{0.4em}%
	\usecounter{enumi}%
	}%
	}%
{\end{list}
}

\renewenvironment{enumerate}{
\begin{myenumerate}}%
{\end{myenumerate}}

\newenvironment{myitemize}{%
\begin{list}{$\bullet$}%
 	{%
	\setlength{\itemsep}{0.4em}%
	\setlength{\topsep}{0.5em}%
	\setlength\leftmargin{2.45em}%
	\setlength\labelwidth{2.05em}%
	\setlength{\labelsep}{0.4em}%
	}%
	}%
{\end{list}}

\renewenvironment{itemize}{
\begin{myitemize}}%
{\end{myitemize}}

\MakePerPage[2]{footnote} 

\title[Random walk bridges conditioned to stay positive]
{An invariance principle for random walk bridges\\conditioned to stay
positive}

\author{Francesco Caravenna}

\address{Dipartimento di Matematica e Applicazioni, Universit\`a
degli Studi di Milano-Bicocca, via Cozzi 53, 20125 Milano, Italy}

\email{francesco.caravenna@unimib.it}

\author{Lo\"ic Chaumont}

\address{LAREMA -- UMR CNRS 6093, Universit\'e d'Angers, 2 bd Lavoisier, 49045 Angers cedex 01}

\email{loic.chaumont@univ-angers.fr}



\thanks{F.C. acknowledges the support of the University of Padova
(through grant CPDA082105/08), of the Laboratoire de Probabilit\'es et Mod\`eles Al\'eatoires 
(during a visit to Paris in the period October-December 2009, at the 
invitation of Giambattista Giacomin and Lorenzo Zambotti)
and of the ERC Advanced Grant VARIS 267356 of Frank den Hollander
(during a visit to Leiden in the period March-May 2012).
Ce travail a b\'en\'eci\'e d'une aide de l'Agence Nationale de la Recherche portant la
r\'ef\'erence ANR-09-BLAN-0084-01.
This research has been supported by the ECOS-CONACYT CNRS Research project M07-M01.}

\date{\today}

\begin{document}

\begin{abstract}
We prove an invariance principle for the bridge of a
random walk conditioned to stay positive, when the random walk
is in the domain of attraction of a stable law, both in the discrete
and in the absolutely continuous setting.
This includes as a special case the convergence under diffusive rescaling
of random walk excursions toward the normalized Brownian excursion,
for zero mean, finite variance  random walks.
The proof exploits a suitable absolute continuity relation together with
some local asymptotic estimates for random walks conditioned
to stay positive, recently
obtained by Vatutin and Wachtel~\cite{cf:VatWac} and Doney~\cite{cf:Don11}. We 
review and extend these relations to the absolutely continuous setting.
\end{abstract}

\maketitle


\section{Introduction}

Invariance principles for conditioned random walks have a long history,
going back at least to the work of Liggett~\cite{cf:Lig}, who proved that the \emph{bridge}
of a random walk in the domain of attraction
of a stable law, suitably rescaled, converges in distribution toward the bridge
of the corresponding stable L\'evy process.
This is a natural extension of Skorokhod's theorem which 
proves the same result for non conditioned random walks, cf.~\cite{cf:Sko57},
itself a generalization to the stable case of Donsker's seminal work~\cite{cf:Donsker}.

Later on, Iglehart~\cite{cf:Igl}, Bolthausen~\cite{cf:Bol76}
and Doney~\cite{cf:Don85} focused on a different type of conditioning:
they proved invariance principles for
random walks \emph{conditioned to stay positive} over a finite time interval,
obtaining as a limit the analogous conditioning for the corresponding L\'evy process,
known as \emph{meander}. More recently, such results have been extended
to the case when the random walk is conditioned to stay positive
\emph{for all time}, cf. Bryn-Jones and Doney \cite{cf:BryDon},
Caravenna and Chaumont \cite{cf:CarCha} and Chaumont and Doney \cite{cf:ChaDon}.

The purpose of this paper is to take a step further, considering the
bridge-type conditioning and the constraint to stay positive at the same time.
More precisely, given a random walk in the domain of attraction
of a stable law, we show that its bridge conditioned to stay positive,
suitably rescaled, converges in distribution toward the bridge of the
corresponding stable L\'evy process conditioned to stay positive.
A particular instance of this result, in the special case of attraction
to the normal law, has recently been obtained in Sohier~\cite{cf:Julien}.
We show in this paper that the result in the general stable case
can be proved exploiting a suitable absolute continuity relation together with
some asymptotic estimates
recently obtained in the literature, cf.~\cite{cf:VatWac} and~\cite{cf:Don11},
that we review and extend.

Besides the great theoretical interest of invariance principles
for conditioned processes, a strong motivation for our results comes from
statistical physics, with particular reference to (1+1)-dimensional polymer and pinning
models interacting with the $x$-axis, cf.~\cite{cf:Gia,cf:Gia2,cf:dH}.
From a mathematical viewpoint, these models may be viewed as perturbations
of the law of a random walk depending on its zero level set. As a consequence,
to obtain the scaling limits of such models, one needs invariance
principles for random walk excursions, that is random walk bridges conditioned
to stay positive that start and end at zero. To the best of our knowledge,
such results were previously known only for simple random walks,
cf.~\cite{cf:Kai}, and were used to obtain the scaling
limits of polymer models in~\cite{cf:CGZ,cf:CGZ2}.
In this paper we deal with bridges that start and end
 at possibly nonzero points, which makes it possible to
deal with polymer models built over non-simple random walks.

The paper is organized as follows.
\begin{itemize}
\item In section~\ref{sec:main} we
state precisely our assumptions and our main results.

\item In section~\ref{sec:fluct} we present some preparatory material on fluctuation theory.

\item Section~\ref{sec:llt}
is devoted to reviewing some important asymptotic estimates for random
walks conditioned to stay positive, in the discrete setting.

\item In section~\ref{sec:lltabs} we extend the above estimates to the absolutely continuous setting.

\item In section~\ref{sec:invpr} we prove the invariance principle.

\item Finally, some more technical details are deferred to the appendices.
\end{itemize}


\section{The invariance principle}
\label{sec:main}


\subsection{Notation and assumptions}
We set $\N := \{1,2,3,\ldots\}$ and $\N_0 := \N \cup \{0\}$.
Given two
positive sequences $(b_n)_{n\in\N}$, $(c_n)_{n\in\N}$, we write as usual
$b_n \sim c_n$ if $\lim_{n\to\infty} b_n/c_n = 1$,
$b_n = o(c_n)$ if $\lim_{n\to\infty} b_n/c_n = 0$ and
$b_n = O(c_n)$ if $\limsup_{n\to\infty} b_n/c_n < \infty$.

We recall that a positive sequence $(b_n)_{n\in\N}$
--- or a real function $b(x)$ --- is said to be \emph{regularly varying
with index $\gamma \in \R$}, denoted $(b_n)_{n\in\N} \in R_\gamma$,
if $b_n \sim n^{\gamma} \ell(n)$, where $\ell(\cdot)$ is a slowly varying
function, i.e. a positive real function
with the property that $\ell(cx)/\ell(x) \to 1$ as
$x\to+\infty$ for all fixed $c > 0$, cf.~\cite{cf:BinGolTeu} for more details.

\smallskip

Throughout this paper we deal with random walks $(S = \{S_n\}_{n\in \N_0}, \bbP)$
in the domain of attraction of a stable law. Let us write precisely this assumption.

\begin{hypothesis} \label{hyp:main}
We assume that $(S = \{S_n\}_{n\in \N_0}, \bbP)$ is
a real random walk in the domain of attraction of a (strictly) stable law
with index $\ga \in (0,2]$ and positivity parameter $\rho \in (0,1)$.
More precisely, we assume that $S_0 = 0$,
the real random variables $\{S_n - S_{n-1}\}_{n\in\N}$ are independent and
identically distributed (i.i.d.) and there exists a sequence $(a_n)_{n\in\N} \in R_{1/\ga}$
such that $S_n/a_n \Rightarrow X_1$, where $(X = \{X_t\}_{t \ge 0}, \bP)$ denotes a
stable L\'evy process with index $\ga \in (0,2]$ and positivity parameter $\rho \in (0,1)$.
\end{hypothesis}

Note that, given a random walk $(S, \bbP)$ satisfying this Hypothesis,
the limiting stable L\'evy process $(X, \bP)$ is determined only up to a multiplicative
constant. In fact, the norming sequence $a_n$ can
be multiplied by any positive constant without affecting Hypothesis~\ref{hyp:main}.

We recall the general constraint $1-\frac{1}{\alpha} \le \rho \le \frac{1}{\alpha}$
(for $\alpha \in (1,2]$, of course).
We also stress that (for $\alpha \in (0,1]$) we assume that $0 < \rho < 1$,
i.e., we exclude subordinators and cosubordinators.
The \emph{Brownian case} corresponds
to $\alpha = 2$, $\rho = \frac 12$, when the limiting L\'evy process $X$ is
(a constant times) Brownian motion. This contains the important special instance
when $\{S_n - S_{n-1}\}_{n\in\N}$ are i.i.d. zero-mean, finite-variance
random variables (the so-called \emph{normal domain of attraction of the
normal law}).

\smallskip

Let us denote by $\Omega^{RW} := \R^{\N_0}$ the discrete paths space
and by $\Omega := D([0,\infty),\R)$ the space of real-valued
c\`adl\`ag paths on $[0,\infty)$, equipped with the Skorokhod
topology, which turns it into a Polish space,
and with the corresponding Borel $\sigma$-field.
We also set $\Omega^{RW}_N := \R^{\{0, \ldots, N\}}$ and $\Omega_t := D([0,t],\R)$.
For notational simplicity,
we assume that $\bbP$ is a law on $\Omega^{RW}$
and $S = \{S_n\}_{n\in\N_0}$ is the coordinate process on this space;
we also denote by $\bbP_x$ the law of the random walk started at $x\in\R$,
i.e. the law on $\Omega_{RW}$ of $S+x$ under $\bbP$.
Analogously, we assume that $X = \{X_t\}_{t\in [0,\infty)}$
is the coordinate process on $\Omega$, that $\bP$ is a law on $\Omega$ and
we denote by $\bP_a$
the law on $\Omega$ of $X+a$ under $\bP$, for all $a\in\R$.
Finally, for every $N\in\N$ we define the rescaling map
$\phi_N : \Omega^{RW} \to \Omega$ by
\begin{equation} \label{eq:phiN}
	\big( \phi_N(S) \big)(t) := \frac{S_{\lfloor Nt\rfloor}}{a_N} \,,
\end{equation}
where $(a_N)_{N\in\N}$ is the norming sequence appearing
in Hypothesis~\ref{hyp:main}. We still denote by $\phi_N$ the
restriction of this map from $\Omega^{RW}_{Nt}$ to $\Omega_t$, for any $t > 0$.


Given $N\in \N$ and $x,y \in [0,\infty)$, by the (law of the)
\emph{random walk bridge of length $N$, conditioned to stay positive,
starting at $x$ and ending at $y$}, we mean either of the following laws on $\Omega^{RW}_N$:
\begin{align} \label{eq:bbPN+0}
	\bbP_{x,y}^{\uparrow,N}(\,\cdot\,) & :=
	\bbP_x( \,\cdot\, | S_1 \ge 0, \ldots, S_{N-1} \ge 0, S_N = y) \,, \\
	\label{eq:hatbbPN+0}
	\widehat\bbP_{x,y}^{\uparrow,N}(\,\cdot\,) & :=
	\bbP_x( \,\cdot\, | S_1 > 0, \ldots, S_{N-1} > 0, S_N = y) \,.
\end{align}
In order for the conditioning in the right hand sides of
\eqref{eq:bbPN+0} and \eqref{eq:hatbbPN+0} to be well-defined,
we work in the lattice or in the absolutely continuous setting. More precisely:

\begin{hypothesis}\label{hyp:main2}
We assume that either of the following assumptions hold:
\begin{itemize}
\item \emph{(lattice case)} The law of $S_1$ under $\bbP$ is supported
by $\Z$ and is aperiodic (i.e. it is not supported by $a \Z + b$, for any
$a \ge 2$ and $b\in\Z$).

\item \emph{(absolutely continuous case)}
The law of $S_1$ under $\bbP$ is absolutely continuous
with respect to the Lebsegue measure on $\R$, and
there exists $n\in\N$ such that
the density $f_n(x) := \bbP(S_n \in \dd x)/\dd x$ of $S_n$ is essentially bounded
(i.e., $f_n \in L^\infty$).
\end{itemize}
\end{hypothesis}

\noindent
Plainly, in the absolutely continuous case we have
$\bbP_{x,y}^{\uparrow,N} = \widehat\bbP_{x,y}^{\uparrow,N}$. We observe that
the requirement that $f_n \in L^\infty$ for some $n\in\N$ is the
standard necessary and sufficient condition
for the uniform convergence of the rescaled density
$x \mapsto a_n f_n(a_n x)$ toward the density of $X_1$,
cf. \cite[\S 46]{cf:GneKol}.
Let us also stress that the aperiodicity assumption in the lattice case is made just for ease of notation:
everything carries through to the periodic case.

Coming back to relations \eqref{eq:bbPN+0} and \eqref{eq:hatbbPN+0},
for the laws $\bbP_{x,y}^{\uparrow,N}$ and
$\widehat \bbP_{x,y}^{\uparrow,N}$ to be well-defined,
in the lattice case we need that the conditioning event
has positive probability. Analogously, in the absolutely continuous case
we require the strict positivity of the density of $S_N$ at $y$
under $\bbP_x$ and under the positivity constrain:
more precisely, denoting by $f(\cdot) = f_1(\cdot)$
the density of the random walk step $S_1$, we need that
\begin{equation}\label{eq:abscontden+}
\begin{split}
	& f_N^{+}(x,y) :=
	\frac{\bbP_x( S_1 > 0, \ldots, S_{N-1} > 0, S_N \in \dd y)}{\dd y} \\
	& = \int_{\{s_1 > 0, \ldots, s_{N-1} > 0\}}
	\left[ f(s_1 - x)  \left( \prod_{i=2}^{N-1} f(s_i - s_{i-1}) \right)
	f(y - s_{N-1}) \right] \dd s_1 \cdots \dd s_{N-1} > 0 \,.
\end{split}
\end{equation}
As a matter of fact, these conditions
will always be satisfied in the regimes for $x,y$ that we consider,
as it will be clear from the asymptotic estimates that we are going to derive.

\smallskip

Next, for $t \in (0,\infty)$ and $a,b \in [0,\infty)$,
we denote by $\bP^{\uparrow,t}_{a,b}$ the law on $\Omega_t$ corresponding to the
\emph{bridge of the L\'evy process of length $t$, conditioned to stay
positive, starting at $a$ and ending at $b$}. Informally, this law is defined
in analogy with \eqref{eq:bbPN+0} and \eqref{eq:hatbbPN+0}, that is
\begin{equation*}
	\bP^{\uparrow,t}_{a,b}(\,\cdot\,) :=
	\bP_a(\,\cdot\,| X_s \ge 0 \ \forall s \in [0,t],\ X_t = b) \,,
\end{equation*}
but we stress that some care is required to give to this definition a proper meaning,
especially in the case when either $a=0$ or $b=0$ (we refer to
section~\ref{sec:invpr} for the details).
We point out that in the Brownian case $\alpha = 2$, $\rho = \frac{1}{2}$,
when $X$ is a standard Brownian motion, $\bP^{\uparrow,1}_{0,0}$ is the
law of the so-called \emph{normalized Brownian excursion}.

\begin{remark}\rm
Let $\bbP_{x}^\uparrow$ and $\widehat\bbP_{x}^\uparrow$ denote
respectively the laws on $\Omega^{RW}$
of the random walks $(S,\bbP_{x})$ and $(S,\widehat\bbP_{x})$ conditioned
to stay positive \emph{for all time}, as defined in \cite{cf:BerDon}
(cf. also \S\ref{srwn} below). The laws
$\bbP_{x,y}^{\uparrow,N}$ and $\widehat\bbP_{x,y}^{\uparrow,N}$ may
be viewed as bridges of $\bbP_x^\uparrow$ and $\widehat\bbP_x^\uparrow$
respectively, i.e.
\begin{equation} \label{eq:abscont1}
	\bbP_{x,y}^{\uparrow,N}(\,\cdot\,) = \bbP_x^\uparrow(\,\cdot\,|S_N=y) \,,
	\qquad
	\widehat\bbP_{x,y}^{\uparrow,N}(\,\cdot\,) =
	\widehat\bbP_x^\uparrow(\,\cdot\,|S_N=y) \,.
\end{equation}
Similarly, if $\bP_a^\uparrow$ is the law on $\Omega$ of
the L\'evy process $(X,\bP_a)$ conditioned to stay positive
for all time, as it is defined in \cite{cf:ChaDon05} (cf. also \S\ref{sec:levypos} below), then
$\bP^{\uparrow,t}_{a,b}$ may be viewed as the bridge of $\bP_a^\uparrow$, i.e.
\begin{equation} \label{eq:abscont2}
	\bP^{\uparrow,t}_{a,b}(\,\cdot\,) = \bP_a^\uparrow(\,\cdot\,|X_t = b) \,.
\end{equation}
In other words, instead of first taking the bridge of a random walk,
or a L\'evy process, and then conditioning
it to stay positive, one can first condition the process to stay positive
(for all time) and then consider its bridge, the resulting process being the same.
\end{remark}

\subsection{The invariance principle}

Recalling the definition \eqref{eq:phiN} of the
(restricted) map $\phi_N : \Omega^{RW}_N \to \Omega_1$,
we denote by $\bbP_{x,y}^{\uparrow,N} \circ \phi_N^{-1}$ the law on
$\Omega_1= D([0,1], \R)$
given by the push-forward of $\bbP_{x,y}^{\uparrow,N}$ through $\phi_N$,
and analogously for $\widehat\bbP_{x,y}^{\uparrow,N}$.

If $\{x_N\}_{N\in\N}$ is a sequence in $\N_0$ such that $x_N/a_N \to a$ as $N\to\infty$,
with $a \ge 0$, it is well known \cite{cf:Sko57} that
\begin{equation} \label{eq:inv}
	\bbP_{x_N} \circ \phi_N^{-1} \Longrightarrow \bP_a \,,
\end{equation}
where ``$\Longrightarrow$'' denotes weak convergence.
Moreover, in \cite{cf:CarCha} we have proved that
\begin{equation} \label{eq:inv+}
	\bbP_{x_N}^\uparrow \circ \phi_N^{-1} \Longrightarrow \bP_a^\uparrow \,.
\end{equation}
Our first result asserts that such an invariance principle also holds for bridges
conditioned to stay positive or non-negative.

\begin{theorem}\label{th:main}
Assume that Hypothesis~\ref{hyp:main} and~\ref{hyp:main2} are satisfied.
Let $a, b \in [0,\infty)$ and let $(x_N)_{N\in\N}, (y_N)_{N\in\N}$
be two non-negative sequences such that $x_N/a_N \to a$ and $y_N / a_N \to b$
as $N\to\infty$ (in the lattice case, we assume that
$x_N, y_N \in \N_0$ for all $N \in\N$). Then as $N\to\infty$
\begin{equation} \label{eq:main}
	\bbP_{x_N,y_N}^{\uparrow,N} \circ \phi_N^{-1}
	\Longrightarrow \bP_{a,b}^{\uparrow,1} \,, \qquad
	\widehat\bbP_{x_N,y_N}^{\uparrow,N} \circ \phi_N^{-1}
	\Longrightarrow \bP_{a,b}^{\uparrow,1} \,.
\end{equation}
\end{theorem}

\noindent
Let us note that, by an easy scaling argument,
this invariance principle immediately generalizes to bridges of any time length.

In the Brownian case $\alpha = 2$, $\rho = \frac{1}{2}$,
the process $\bP^{\uparrow, t}_{a,b}$ has continuous paths.
In this situation, it is standard to pass from
weak convergence on $D([0,1],\R)$ to weak convergence
on $C([0,1],\R)$, the space of
real-valued continuous functions defined on $[0,1]$, endowed with the topology
of uniform convergence and with the corresponding Borel $\sigma$-field.
For this purpose, we introduce the map $\psi_N : \Omega^{RW}_N
\to C([0,1],\R)$, analogous to $\phi_N$ defined in \eqref{eq:phiN},
but corresponding to \emph{linear interpolation}, i.e.
\begin{equation*}
	\big(\psi_N(S)\big)(t) \,:=\, \frac{(1 + \lfloor Nt \rfloor - Nt)S_{\lfloor Nt\rfloor}
	+ (Nt- \lfloor Nt \rfloor)S_{\lfloor Nt\rfloor + 1}}{a_N} \,,
	\qquad \forall t \in [0,1]\,.
\end{equation*}
For ease of reference, we state an important special case of Theorem~\ref{th:main}.

\begin{corollary}[Brownian case]
Let $(\{X_n\}_{n\in\N}, \bbP)$ be i.i.d. real random variables with zero mean and unit
variance and let $S_0 = 0$, $S_n = S_{n-1} + X_n$, $n\in\N$,
be the associated random walk, so that Hypothesis~\ref{hyp:main} is satisfied
with $a_N = \sqrt{N}$ and $X$ a standard Brownian motion.
Assume that Hypothesis~\ref{hyp:main2} is satisfied
and let $(x_N)_{N\in\N}$, $(y_N)_{N\in\N}$ be non-negative sequences that are
both $o(\sqrt{N})$ (in the lattice case, also assume that
$x_N, y_N \in \N_0$ for all $N \in\N$).

Then, as $N\to\infty$,
\begin{equation*}
	\bbP_{x_N,y_N}^{\uparrow,N} \circ \psi_N^{-1} \Longrightarrow \bP_{0,0}^{\uparrow,1}\,,
	\qquad
	\widehat\bbP_{x_N,y_N}^{\uparrow,N} \circ \psi_N^{-1}
	\Longrightarrow \bP_{0,0}^{\uparrow,1} \,.
\end{equation*}
In words: the random walk bridge of length $N$
conditioned to stay positive or non-negative, starting at $x_N$
and ending at $y_N$, under linear interpolation and
diffusive rescaling, converges in distribution
on $C([0,1],\R)$ toward the normalized Brownian excursion.
\end{corollary}

The proof of Theorem~\ref{th:main} bears on the absolute continuity
of $\bbP_{x_N,y_N}^{\uparrow,N}$ with respect to
$\bbP_{x_N}^\uparrow$, cf. \eqref{eq:abscont1},
and exploits the convergence \eqref{eq:inv+}. In order to apply
these arguments, we need a uniform control
of the Radon-Nikodym density of $\bbP_{x_N,y_N}^{\uparrow,N}$
with respect to $\bbP_{x_N}^\uparrow$. This requires precise local estimates
of the kernel
$f_N^+(x,y)$, cf. \eqref{eq:abscontden+}, in the absolutely continuous case,
and of the analogous kernels $q_N^+(x,y)$ and $\widehat q_N^+(x,y)$ in the lattice case:
\begin{equation} \label{eq:q}
\begin{split}
	q_N^+(x,y) & := \bbP_x(\widehat \tau_1^- \ge N,\, S_N = y) =
	\bbP_x(S_1 \ge 0, \ldots, S_{N-1} \ge 0, S_N = y) \,, \\
	\widehat q_N^+(x,y) & := \bbP_x(\tau_1^- \ge N,\, S_N = y) =
	\bbP_x(S_1 > 0, \ldots, S_{N-1} > 0, S_N = y) \,.
\end{split}
\end{equation}

In the lattice case, such local limit theorems have been
proved by Vatutin and Wachtel~\cite{cf:VatWac} and Doney~\cite{cf:Don11}
and are reviewed in Proposition \ref{th:xq}.
The proof of the local limit theorems
for $f_N^+(x,y)$, in the absolutely continuous case,
is the second main result of this paper, cf. Theorem~\ref{th:xqabscont}
in section~\ref{sec:lltabs}. This is obtained from the Stone version of the local limit theorems,
also proved in \cite{cf:VatWac,cf:Don11},
through a careful approximating procedure.

We point out that our approach differs from that
of Sohier \cite{cf:Julien}
in the Brownian case, where the weak convergence of the sequence
$\bbP_{x_N,y_N}^{\uparrow,N}$ is established, in a more classical way, proving
tightness and convergence of the finite dimensional distributions.

\begin{remark}\rm
For the asymptotic behavior of $f_N^+(x,y)$ in the absolutely
continuous case, we need a suitable condition, linked to \emph{direct Riemann integrability}, on
a convolution of the random walk step density $f(\cdot)$
(cf. section~\ref{sec:lltabs}
for details, in particular \eqref{eq:dRihyp}).
This is a very mild condition, which is immediately checked if, e.g.,
there exist $C > 0$, $n\in\N$ and $\epsilon > 0$ small enough
such that $|f_n(x)| \le C/|x|^{1+\alpha - \epsilon}$
for every $x \in \R$. As a matter of fact, it turns out that
this condition is automatically satisfied with no further assumption beyond
Hypotheses~\ref{hyp:main} and~\ref{hyp:main2}, as it is proved in~\cite{cf:Car2}.
\end{remark}

\section{Preparatory Material}

\label{sec:fluct}

\subsection{An important notation on sequences}
\label{sec:notation}

We will frequently deal with sequences $(b_n(z))_{n\in\N}$ indexed by a real parameter $z$.
Given a family of subsets $V_n \subseteq \R$, we write
\begin{equation} \label{eq:unifnot}
	\text{``}b_n(z) = o(1) \ \ \text{uniformly for } z \in V_n\text{''}
	\qquad \text{to mean} \qquad \lim_{n\to\infty} \, \sup_{z \in V_n} |b_n(z)| = 0 \,.
\end{equation}
We stress that this is actually equivalent to the seemingly weaker condition
\begin{equation*}
	\lim_{n\to\infty} b_n(z_n) = 0  \ \ \text{for any fixed sequence $(z_n)_{n\in\N}$
	such that $z_n \in V_n$ for all $n\in\N$}\,,
\end{equation*}
as one checks by contradiction.
We also note that, by a subsequence argument,
to prove such a relation
it is sufficient to consider sequences $(z_n)_{n\in\N}$
(such that $z_n \in V_n$ for all $n\in\N$) that converge to a
(possibly infinite) limit, i.e. such that $z_n \to c \in \R \cup \{\pm\infty\}$.

Given $(b_n(z))_{n\in\N}$, with $z \in \R$,
and a fixed positive sequence $(a_n)_{n\in\N}$,
it is sometimes customary to write
\begin{align} \label{eq:blablabla}
	\text{``$b_n(z) = o(1)$ \ uniformly for $z = o(a_n)$''}
\end{align}
as a shorthand for
\begin{equation} \label{eq:blablabla2}
	\text{$b_n(z) = o(1)$ \ uniformly for $z \in [0,\epsilon_n]$,
	for any fixed sequence $\epsilon_n = o(a_n)$} \,.
\end{equation}
Again, this is equivalent to the apparently weaker statement
\begin{align} \label{eq:blablasimple}
	\text{$b_n(z_n) = o(1)$ \ for any fixed sequence $z_n = o(a_n)$} \,,
\end{align}
as an easy contradiction argument shows. The formulation \eqref{eq:blablabla}--\eqref{eq:blablabla2}
is usually preferred
when stating and applying theorems, while \eqref{eq:blablasimple} is nicer to handle when proving them.

In the sequel, we sometimes write $(const.)$,
$(const.')$ to denote generic positive constants, whose value
may change from place to place.

\subsection{Fluctuation theory for random walks}
\label{sec:fluctutation}

For the purpose of this subsection, we only assume that
$(S=\{S_n\}_{n\ge 0}, \bbP)$ is a real random walk starting at zero: this means that
$S_0 = 0$ a.s. and $(\{S_n - S_{n-1}\}_{n\ge 1}, \bbP)$ are i.i.d. real random
variables. To avoid degeneracies, we assume that the walk is not constant, i.e.,
$\bbP(S_1 = c)<1$ for all $c\in\R$.

We denote by $\{\tau_k^\pm\}_{k\ge 0}$ and $\{H_k^\pm\}_{k\ge 0}$ the \emph{weak}
ascending ($+$) and descending ($-$) ladder epoch and ladder height processes respectively,
that is $\tau_0^\pm := 0$, $H_0^\pm := 0$ and for $k\ge 1$
\begin{equation*}
	\tau_{k}^\pm := \inf \big\{ n > \tau_{k-1}^\pm :\ \pm S_n \ge \pm S_{\tau_{k-1}^\pm} \big\} \,,
	\qquad H_k^\pm := \pm S_{\tau_k^\pm} \,.
\end{equation*}
Note that $H_k^-$ is a non-negative random variable.
In fact, $\{\tau_k^\pm\}_{k\ge 0}$ and $\{H_k^\pm\}_{k\ge 0}$
are \emph{renewal processes}, i.e., random walks with
i.i.d. non-negative increments.
The \emph{strict} ascending ($+$) and descending ($-$)
ladder epoch and ladder height processes
$\{\widehat\tau_k^\pm\}_{k\ge 0}$ and $\{\widehat H_k^\pm\}_{k\ge 0}$ are defined replacing
the weak inequality $\ge$ by the strict one $>$ in the preceding display.

We set $\zeta := \bbP(H_1^+ = 0) \in [0,1)$ and we note
that we also have $\zeta = \bbP(H_1^- = 0)$. In fact for all $n\in\N$ we can write
\begin{equation*}
\begin{split}
	\bbP(\tau_1^-=n, H_1^-=0) & = \bbP(S_1 > 0, \ldots, S_{n-1}>0, S_n=0) \\
	& = \bbP(S_1<0, \ldots, S_{n-1}<0, S_n=0) = 	\bbP(\tau_1^+=n, H_1^+=0) \,,
\end{split}
\end{equation*}
by time reversal, i.e., observing that
the process $\{S_{n}-S_{n-k}\}_{0 \le k \le n}$
has the same law as $\{S_{k}\}_{0 \le k \le n}$. It is also easy to check that
the laws of $H_1^\pm$ and $\widehat H_1^\pm$ are closely related:
\begin{equation*}
	\bbP(H_1^\pm \in \dd x) = \zeta \gd_{0}(\dd x) + (1-\zeta)
	\bbP(\widehat H_1^\pm \in \dd x) \,.
\end{equation*}

We denote by $V^\pm(\cdot)$ the weak ascending ($+$) and descending
($-$) renewal function, defined for $x \ge 0$ by
\begin{equation} \label{eq:defV}
	V^\pm(x) = \bbE \big[ \#\{k\ge 0: H_k^\pm \le x\}\big] = \sum_{k = 0}^\infty
	\bbP(H_k^\pm \le x) = \sum_{k=0}^\infty \sum_{n=0}^\infty
	\bbP(\tau_k^\pm = n, \pm S_n \le x) \,.
\end{equation}
Note that $V^\pm(0) = \sum_{k=0}^\infty \bbP( H_k^\pm = 0)
= \sum_{k=0}^\infty \zeta^k = (1-\zeta)^{-1}$.
Analogously, we can define $\widehat V^\pm(x)$ for $x \ge 0$ by
\begin{equation} \label{eq:defVhat}
	\widehat V^\pm(x) = \bbE \big[ \#\{k\ge 0: \widehat H_k^\pm \le x\}\big]
	= \sum_{k = 0}^\infty \bbP(\widehat H_k^\pm \le x) = \sum_{k=0}^\infty \sum_{n=0}^\infty
	\bbP(\widehat \tau_k^\pm = n, \pm S_n \le x) \,,
\end{equation}
and note that $\widehat V^\pm(0) = 1$. As a matter of fact,
the following relation holds:
\begin{equation} \label{eq:VhatV}
	\widehat V^\pm(x) = (1-\zeta) V^\pm(x)\,, \qquad \forall x \ge 0 \,,
\end{equation}
(cf. equation~(1.13) in \cite[\S~XII.1]{cf:Fel2}),
therefore working with $V^\pm(x)$ or $\widehat V^\pm(x)$ is equivalent.
Note that both functions $V^\pm(\cdot)$ and $\widehat V^\pm(\cdot)$ are
non-decreasing and right-continuous.

Finally, we introduce two \emph{modified} renewal functions $\underline V^\pm(\cdot)$
and $\underline {\widehat V}^\pm(\cdot)$. For strictly positive values
of $x > 0$ we set
\begin{gather} \label{eq:defV-}
	\underline V^\pm(x) = \bbE \big[ \#\{k\ge 0: H_k^\pm < x\}\big]
	= \sum_{k = 0}^\infty \bbP( H_k^\pm < x) = \sum_{k = 0}^\infty
	\sum_{n=0}^\infty \bbP( \tau_k^\pm = n, \pm S_n < x) \,, \\
	\label{eq:defV-hat}
	\underline {\widehat V}^\pm(x) = \bbE \big[ \#\{k\ge 0: \widehat H_k^\pm < x\}\big]
	= \sum_{k = 0}^\infty \bbP( \widehat H_k^\pm < x) = \sum_{k = 0}^\infty
	\sum_{n=0}^\infty \bbP( \widehat \tau_k^\pm = n, \pm S_n < x) \,.
\end{gather}
Note that that these functions
are \emph{left-continuous} on $(0,\infty)$.
They can be recovered from the previously introduced ones
through a simple limiting procedure: for every $x > 0$
\begin{equation} \label{eq:leftco}
	\underline V^\pm(x) = V^\pm(x-) :=
	\lim_{\gep \downarrow 0}V^\pm(x-\gep) \,,
	\quad \ \
	\underline {\widehat V}^\pm(x) = \widehat V^\pm(x-) :=
	\lim_{\gep \downarrow 0} {\widehat V}^\pm(x-\gep) \,.
\end{equation}
We stress that in the lattice case, when the law of $S_1$ is supported by $\Z$,
we simply have $\underline V^\pm(x) = V^\pm(x-1)$
and $\underline {\widehat V}^\pm(x) = {\widehat V}^\pm(x-1)$ for every $x \in \N$
($x \ge 1$).
We complete the definition of the functions $\underline V^\pm(x)$
and $\underline {\widehat V}^\pm(x)$ by setting, for $x = 0$,
\begin{equation} \label{eq:VV+0}
	\underline V^\pm(0) := 1 \,, \qquad
	\underline {\widehat V}^\pm(0) := 1-\zeta = \bbP(H_1^\pm > 0) \,,
\end{equation}
so that a relation completely analogous to \eqref{eq:VhatV} holds:
\begin{equation} \label{eq:underVhatV}
	\underline {\widehat V}^\pm(x) = (1-\zeta)
	{\underline V}^\pm(x)\,, \qquad \forall x \ge 0 \,.
\end{equation}

The reason for introducing the modified renewal functions $\underline V^\pm(\cdot)$
and $\underline {\widehat V}^\pm(\cdot)$
is explained by the following Lemma, proved in Appendix~\ref{sec:proofharm}.
We recall that $\bbP_x$ denotes the law of $S$ started at $x \in \R$, that is
$\bbP_x(S \in \cdot) = \bbP(S+x \in \cdot)$.


\smallskip
\begin{lemma} \label{th:harm}
Assume that the random walk
$(\{S_n\}_{n\ge 0}, \bbP)$ does not drift to $-\infty$, that is
$\limsup_k S_k = + \infty$, $\bbP$--a.s.. Then the function
$V^-(\cdot)$ is invariant for the semigroup of
the random walk killed when it first enters the negative half-line
$(-\infty, 0)$, i.e.,
\begin{equation} \label{eq:V}
    V^-(x) = \bbE_x \big( V^-(S_N)\, \ind_{\{S_1 \ge 0, \ldots, S_N \ge 0\}} \big)
    \qquad \forall x \ge 0 , N \in \N\,.
\end{equation}
Analogously, the function $\underline V^-(x)$ is invariant for the semigroup of
the random walk killed when it first enters the non-positive half-line
$(-\infty, 0]$:
\begin{equation} \label{eq:V-}
	\underline V^-(x) = \bbE_x \big( \underline V^-(S_N)\, \ind_{\{S_1 > 0, \ldots, S_N > 0\}} \big)
    \qquad \forall x \ge 0, N \in \N\,.
\end{equation}
\end{lemma}
\smallskip

We note that by the symmetry
$S \to -S$ it follows immediately from \eqref{eq:V} and \eqref{eq:V-} that
if the random walk $(\{S_n\}_{n\ge 0}, \bbP)$ does not drift to $+\infty$ we have
\begin{align} \label{eq:Valt}
	V^+(x) & = \bbE_{-x} \big( V^+(-S_N)\, \ind_{\{S_1 \le 0, \ldots, S_N \le 0\}} \big)
	\qquad \forall x \ge 0, N \in \N \,, \\
	\label{eq:V-alt}
	\underline V^+(x) & = \bbE_{-x} \big( \underline V^+(-S_N)\,
	\ind_{\{S_1 < 0, \ldots, S_N < 0\}} \big)
	\qquad \forall x \ge 0, N \in \N \,.	
\end{align}
Also note that, in all the relations \eqref{eq:V}, \eqref{eq:V-},
\eqref{eq:Valt}, \eqref{eq:V-alt}, one can replace $V^\pm(\cdot)$
by ${\widehat V}^\pm(\cdot)$ and $\underline V^\pm(\cdot)$
by $\underline {\widehat V}^\pm(\cdot)$ respectively, thanks to
\eqref{eq:VhatV} and \eqref{eq:underVhatV}.

\begin{remark}\rm \label{th:otherproof}
A proof of Lemma~\ref{th:harm} is implicit in part~2.3 of \cite{cf:Ber1}.
The fact that the invariant function
$h^\uparrow$ in that paper coincides with $\underline{\widehat V}^-$
follows from the celebrated \emph{Duality Lemma}, cf. Chapter~XII
in~\cite{cf:Fel2}: for any $n\in\N$
and any Borel subset $I \subset (0, \infty)$
\begin{equation*}
\begin{split}
	\sum_{k\ge 0} \bbP \big( \widehat \tau^\pm_k = n,\, \pm S_n \in I \big)
	= \bbP \big( \pm S_1 > 0,\ldots, \pm S_n > 0 ,\, \pm S_n \in I \big)
	= \bbP \big( \tau^{\mp}_1 > n ,\, \pm S_n \in I \big) \,,\\
	\sum_{k\ge 0} \bbP \big( \tau^\pm_k = n,\, \pm S_n \in I \big)
	= \bbP \big( \pm S_1 \ge 0,\ldots, \pm S_n \ge 0 ,\, \pm S_n \in I \big)
	= \bbP \big( \widehat \tau^{\mp}_1 > n ,\, \pm S_n \in I \big) \,.
\end{split}
\end{equation*}
We also refer to part~2 of~\cite{cf:BerDon}.
\end{remark}


\subsection{Some consequences of our assumptions.}
\label{sec:conseq}

With the notation of Hypothesis~\ref{hyp:main}, let $g(\cdot)$
the density of the random variable
$X_1$. We denote by $g^+(\cdot)$
the density of the time-one marginal distribution of the meander
\cite{cf:Cha97} of the L\'evy process
$X$, which can be defined informally by
$g^{+}(x) \, \dd x = \bP_0(X_{1} \in \dd x \,|\, \inf_{0\le s \le 1}X_{s} \ge 0)$, see Lemma 4 in 
\cite{cf:ChaDon}.
Analogously, $g^-(\cdot)$ is the density of the time-one marginal distribution of the meander
of $-X$.

%

\smallskip

By classical results \cite{cf:GreOmeTeu,cf:DonGre,cf:Don85},
when Hypothesis~\ref{hyp:main} holds, the
random vectors $(\tau_1^+, H_1^+)$ and $(\widehat \tau_1^+,\widehat  H_1^+)$
are in the domain of attraction of a bivariate stable law of
indexes $(\rho, \ga \rho)$ (cf. \cite{cf:ChaDon} for a general version of this result),
in particular
\begin{equation} \label{eq:sumup1}
	\bbP(\tau_1^+ > n) \in R_{-\rho} \,, \qquad
	\bbP(H_1^+ > x) \in R_{-\ga\rho} \,, \qquad
	V^+(x) \in R_{\ga\rho} \,.
\end{equation}
An analogous statement holds
for the descending ladder variables
$(\tau_1^-, H_1^-)$ or $(\widehat \tau_1^-,\widehat  H_1^-)$: it suffices
to replace $\rho$ by $1-\rho$, so that
\begin{equation} \label{eq:sumup2}
	\bbP(\tau_1^- > n) \in R_{-(1-\rho)} \,, \qquad
	\bbP(H_1^- > x) \in R_{-\ga(1-\rho)}\,, \qquad
	V^-(x) \in R_{\ga(1-\rho)} \,.
\end{equation}
Furthermore, the following relation holds:
\begin{equation} \label{eq:usefuly0}
	\bbP(\widehat \tau_1^- > n) \sim (1-\zeta)^{-1} \bbP( \tau_1^- > n) \,,
\end{equation}
as we prove in Appendix~\ref{sec:319}.

We point out that,
by equation~(31) in~\cite{cf:VatWac}, as $n \to\infty$
\begin{equation} \label{eq:lien}
	\underline {\widehat V}^+(a_n)
	\sim {\mathtt C}^+ \, n \, \bbP(\tau_1^- > n) \,,
	\qquad
	\underline {\widehat V}^-(a_n)
	\sim {\mathtt C}^- \, n \, \bbP(\tau_1^+ > n) \,,
\end{equation}
and, by Theorem 1 of \cite{cf:DonSav}, as $\epsilon \downarrow 0$
\begin{equation} \label{eq:DonSav}
	g^+(\epsilon) \sim \tilde {\mathtt C}^+ \, g(0) \, \epsilon^{\alpha\rho} \,,
	\qquad
	g^-(\epsilon) \sim \tilde {\mathtt C}^- \, g(0) \, \epsilon^{\alpha(1-\rho)}\,.
\end{equation}
It turns out that the corresponding constants in the preceding relations coincide, i.e.:
\begin{equation} \label{eq:C+-}
	{\mathtt C}^+ = \tilde {\mathtt C}^+ \,, \qquad \quad
	{\mathtt C}^- = \tilde {\mathtt C}^- \,,
\end{equation}
as we prove in Lemma~\ref{th:CC} below.
We stress that the precise value of these constants is not universal.
In fact, if we change the norming sequence, taking
$a'_n := c a_n$ with $c > 0$, then $S_n / a'_n \Rightarrow X'_1 := X_1/c$
(recall Hypothesis~\ref{hyp:main}):
rewriting \eqref{eq:lien} and \eqref{eq:DonSav} for $a'_n$
and for the densities associated to $X'$, one sees that the constants
${\mathtt C}^\pm$ and $\tilde{\mathtt C}^\pm$
get divided by $c^{\alpha\rho}$.

\section{Local limit theorems in the lattice case}

\label{sec:llt}

In this section we put ourselves in the lattice case, cf. Hypothesis~\ref{hyp:main2},
so we assume that the random
walk law is supported by $\Z$ and is aperiodic.
Our purpose is to give the precise asymptotic behavior as $n\to\infty$
of the kernels $q_n^+(x,y)$ and $\widehat q_n^+(x,y)$,
defined in~\eqref{eq:q}.

When both $x/a_n$ and $y/a_n$ stay away from $0$ and $\infty$,
this is an easy consequence of
Gnedenko's local limit theorem~\cite{cf:GneKol}, which states that as $n\to\infty$
\begin{equation}\label{eq:Gnedenko}
	\bbP_x (S_n = y) = \frac{1}{a_n}
	\bigg\{ g\bigg( \frac{y-x}{a_n} \bigg) + o(1) \bigg\} \,, \qquad
	\text{uniformly for } x,y \ge 0 \,,
\end{equation}
and Liggett's invariance principle
for the bridges~\cite{cf:Lig}. More precisely, setting for $a,b > 0$
\begin{equation*}
	C(a,b) := \bP_a \bigg( \inf_{0 \le s \le 1} X_s \ge 0
	\,\bigg|\, X_1 = b \bigg) \,,
\end{equation*}
for every fixed $\epsilon > 0$ we have as $n\to\infty$
\begin{equation} \label{eq:eeasy}
	q_n^+(x,y) = \frac{1}{a_n}\, g \bigg( \frac{y-x}{a_n} \bigg) \,
	C \left( \frac{x}{a_n},\frac{y}{a_n} \right) (1+o(1)) \,,
	\quad \text{unif. for } x,y \in \left(\gep a_n , \frac{1}{\gep} a_n \right) \,,
\end{equation}
and exactly the same relation holds for $\widehat q_n^+(x,y)$.

Note that $C(a,b)>0$ for all $a,b>0$, but $C(a,b) \to 0$
if $\min\{a,b\} \to 0$, therefore when either $x = o(a_n)$
or $y = o(a_n)$ relation \eqref{eq:eeasy} only says that
$q_n^+(x,y) = o(1/a_n)$, and analogously $\widehat q_n^+(x,y) = o(1/a_n)$.
Precise estimates in this regime have been obtained in the last
years: for ease of reference, we sum them up explicitly
(and somewhat redundantly) in the following proposition.

%

\begin{proposition}\label{th:xq}
Assume that Hypothesis~\ref{hyp:main} and Hypothesis~\ref{hyp:main2} (lattice case)
hold. Then the following relations hold as $n\to\infty$, for $x,y \in \N_0$:
\begin{gather} \label{eq:xqlarge}
	\left. \begin{array}{r}
	\displaystyle
	q_n^+(x,y) = \frac{\bbP(\widehat\tau_1^- > n)}{a_n}
	\, \widehat V^-(x) \left( g^+ \left(\frac{y}{a_n}\right) + o(1) \right) \\
	\rule{0pt}{1.9em}\displaystyle
	\widehat q_n^+(x,y) = \frac{\bbP(\tau_1^- > n)}{a_n}
	\, \underline{V}^-(x) \left( g^+ \left(\frac{y}{a_n}\right) + o(1) \right)
	\end{array}	 \right\}
	\ \text{unif. for $x = o(a_n)$, $y \ge 0$} \,, \ \\
	\label{eq:xqlargeinv}
	\left. \begin{array}{r}
	\displaystyle
	q_n^+(x,y) = \frac{\bbP(\widehat\tau_1^+ > n)}{a_n}
	\, \widehat V^+(y) \left( g^- \left(\frac{x}{a_n}\right) + o(1) \right) \\
	\rule{0pt}{1.9em}\displaystyle
	\widehat q_n^+(x,y) = \frac{\bbP(\tau_1^+ > n)}{a_n}
	\, \underline{V}^+(y) \left( g^- \left(\frac{x}{a_n}\right) + o(1) \right)
	\end{array}	 \right\}
	\ \text{unif. for $y = o(a_n)$, $x \ge 0$} \,, \ \\
	\label{eq:xqsmall}
	\left. \begin{array}{r}
	\displaystyle
	q_n^+(x,y) = \frac{g(0)}{n \, a_n} \, \widehat V^-(x) \, V^+(y) \,
	\big( 1 + o(1) \big) \\
	\rule{0pt}{1.9em}\displaystyle
	\widehat q_n^+(x,y) = \frac{g(0)}{n \, a_n} \, \underline{V}^-(x) \,
	\underline{\widehat V}^+(y) \, \big( 1 + o(1) \big)
	\end{array} \right\}
	\ \text{uniformly for $x = o(a_n)$, $y = o(a_n)$}\,.
\end{gather}
\end{proposition}

As this list of relations may appear intimidating, let us give some directions.
A first notational simplification is that all the different renewal functions
appearing in these relations can be expressed in terms of $V^\pm(\cdot)$.
In fact $\widehat V^\pm (\cdot) = (1-\zeta) V^\pm(\cdot)$
by \eqref{eq:VhatV}, and furthermore
in the lattice case $\underline V^\pm(x) = V^\pm(x-1)$
and $\underline {\widehat V}^\pm(x) = (1-\zeta)V^\pm(x-1)$
for all $x \in \N$ ($x \ge 1$).

Another basic observation is that the relations in \eqref{eq:xqlargeinv} can be immediately
deduced from those in \eqref{eq:xqlarge}: it suffices to
consider the random walk $-S$ instead of $S$ and
to exchange $x$ with $y$ and every ``$+$'' quantity
with the corresponding ``$-$'' one. Therefore it is sufficient to
focus on the relations in \eqref{eq:xqlarge} and \eqref{eq:xqsmall},
but there are further simplifications.
In fact, observing that $\widehat q_n^+(x,y) = q_n^+(x-1,y-1)$
when $x,y \in \N$ (recall \eqref{eq:q}),
the first equation in \eqref{eq:xqsmall} follows from the second one if both $x,y \ge 1$.
Analogously, using relation \eqref{eq:usefuly0},
the first equation
in \eqref{eq:xqlarge} follows from the second one
when both $x,y \ge 1$.

Summarizing, for Proposition~\ref{th:xq} it is sufficient to prove:
\begin{itemize}
\item the four relations
in \eqref{eq:xqlarge} and \eqref{eq:xqsmall} in the special case $x = 0$;

\item the second relations in \eqref{eq:xqlarge} and \eqref{eq:xqsmall}
for general $x$.
\end{itemize}
The second relations in \eqref{eq:xqlarge}
and \eqref{eq:xqsmall} for $x=0$ were proved
in~\cite[Theorem~5 and Theorem~6]{cf:VatWac},
while the first ones can be deduced arguing as in
page~100 of~\cite{cf:AliDon}.\footnote{Previously,
these relations were proved in the gaussian case
($\ga = 2$, $\rho = \frac 12$), cf. \cite[Proposition~1]{cf:BryDon} and
\cite[Theorem~4]{cf:Car} for \eqref{eq:xqlarge} and
\cite[equation~(9)]{cf:AliDon} for \eqref{eq:xqsmall}.}
As a matter of fact,
the case $x=y=0$ of \eqref{eq:xqsmall} has
not been considered in \cite{cf:VatWac}, but it can be easily
deduced, as we show in Appendix~\ref{sec:45-2} and~\ref{sec:45-1}.
The second relations in \eqref{eq:xqlarge}
and \eqref{eq:xqsmall} for general $x$ have been recently proved
in~\cite[Proposition~11]{cf:Don11}, using a decomposition that allows
to express them as a function of the $x=0$ case.
This completes the justification of Proposition~\ref{th:xq}.



\begin{remark}\rm
Since the functions $g^\pm(z)$ vanish both when $z \downarrow 0$
and when $z \to +\infty$, the two relations in \eqref{eq:xqlarge} give
the \emph{precise} asymptotic behavior
(i.e., the ratio of the two sides of the equation converges to 1) only when $y/a_n$ is bounded away
from $0$ and $\infty$. When $y/a_n \to 0$, that is
$y = o(a_n)$, the precise asymptotic behavior
is given by \eqref{eq:xqsmall}.
When $y/a_n \to +\infty$ and $\alpha \rho < 1$ (which excludes the Brownian case),
the precise asymptotic behavior can be derived
under additional assumptions,
cf. Proposition~13 in~\cite{cf:Don11}.
\end{remark}

\begin{remark}\rm \label{rem:onlygauss?}
In the gaussian case ($\ga = 2$, $\rho = \frac12$)
several explicit expressions are available. For instance,
$g^{\pm}(x) = x \, e^{-x^{2}/2} \, \ind_{(0,\infty)}(x)$
and $g(x) = (2\pi)^{-1/2} e^{-x^{2}/2}$, hence the constants
in \eqref{eq:DonSav} are $\widetilde{\mathtt{C}}^\pm = \sqrt{2\pi}$.
From equation~(2.6) and the last equation on p.~515 in~\cite{cf:Car},
it follows that also for the constants in \eqref{eq:lien}
one has ${\mathtt C}^\pm = \sqrt{2\pi}$, in agreement
with \eqref{eq:C+-}.

Furthermore, since $V(\cdot) \in R_{1}$,
if $y/a_{n}$ is bounded away from $0$ and $\infty$ we can write
$\underline {\widehat V}^{+}(y) \sim \frac{y}{a_{n}} \underline {\widehat V}^{+}(a_{n})$.
It follows that the second relations in \eqref{eq:xqlarge} and \eqref{eq:xqsmall}
can be gathered in the following single one:
\begin{equation}\label{eq:hatq}
	\widehat q_n^+(0,y) \sim \frac 1n \, \frac{1}{a_n} \, g\left(\frac{y}{a_{n}}\right) \,
	\underline {\widehat V}^+(y) \sim \frac 1n \, \bbP(S_{n} = y) \,
	\underline {\widehat V}^+(y) \,,
\end{equation}
which holds uniformly in $y \in [0,Ma_{n}]$, for any fixed $M > 0$
(cf. equation~(1.7) in~\cite{cf:Car}).

It is natural to ask whether relation \eqref{eq:hatq}
still holds for $\alpha < 2$.
Recalling \eqref{eq:lien}, \eqref{eq:DonSav}
and \eqref{eq:C+-}, this is equivalent to asking whether
\eqref{eq:DonSav} can be strengthened to
\begin{equation} \label{eq:rell}
	g^{+}(x) = \widetilde{\mathtt C}^+ x^{\ga\rho} g(x) \ind_{(0,\infty)}(x) \,,
	\qquad \forall x \in \R\,.
\end{equation}
Arguing as in \cite{cf:AliCha}, it is not difficult to
show that this relation holds when the limiting L\'evy process has no positive
jumps, i.e. for $\alpha \in (1,2)$ and $\rho = 1 - 1/\alpha$.
We conjecture that relation \eqref{eq:rell} fails whenever
$\rho \ne 1-1/\alpha$. In the symmetric
Cauchy case $\alpha = 1$, $\rho = 1/2$ it has been shown
that indeed relation \eqref{eq:rell} does not hold
(cf. the comments following Proposition~1 in~\cite{cf:AliCha}).
\end{remark}

\smallskip

We conclude the section proving that the constants
${\mathtt C}^\pm$ and $\tilde {\mathtt C}^\pm$ indeed coincide.

\begin{lemma}\label{th:CC}
Recalling relations \eqref{eq:lien}
and \eqref{eq:DonSav}, we have ${\mathtt C}^+ = \tilde {\mathtt C}^+$
and ${\mathtt C}^- = \tilde {\mathtt C}^-$.
\end{lemma}

\noindent
For the proof we need the following very general result.

\begin{lemma}\label{th:remarkable}
Let $\{h_n(\cdot)\}_{n\in\N}$ be an arbitrary sequence of real functions,
all defined on the same subset $I \subseteq \R$. Assume that for
every $z \in I$ and for every sequence $\{z_n\}_{n\in\N}$ of $I$, such that $z_n \to z$,
the limit $h(z) := \lim_{n\to\infty} h_n(z_n)$
exists and does not depend on the sequence $\{z_n\}_{n\in\N}$,
but only on the limit point $z$. Then the function $h: I \to \R$ is continuous.
\end{lemma}

\begin{proof}
We proceed by contradiction. If $h$ is not continuous, there exist
$\epsilon > 0$, $\bar z \in I$
and a sequence $\bar z^{(k)} \to \bar z$ such that
$|h(\bar z) - h(\bar z^{(k)})| > 2\epsilon$ for every $k\in\N$.
For every fixed $k \in \N$ we have $\lim_{n\to\infty} h_n(\bar z^{(k)}) = h(\bar z^{(k)})$
by assumption, hence there exists $\tilde n(k) \in \N$ such that
$|h(\bar z^{(k)}) - h_{\tilde n(k)}(\bar z^{(k)})| < \epsilon$.
By the triangle inequality,
we then have
\begin{equation} \label{eq:toolest}
	|h_{\tilde n(k)}(\bar z^{(k)}) - h(\bar z)| \;\ge\;
	|h(\bar z) - h(\bar z^{(k)})| \,-\,
	|h(\bar z^{(k)}) - h_{\tilde n(k)}(\bar z^{(k)})| \;>\; \epsilon \,, \quad
	\forall k \in \N\,.
\end{equation}
Observe that $\tilde n(k)$ can be taken as large as we wish, hence
we may assume that $k \mapsto \tilde n(k)$ is increasing. We also set $\tilde n(0) := 0$
for convenience. Let us finally define the sequence $\{z_n\}_{n\in\N}$ by
$z_n := \bar z^{(k)}$, where $k \in \N$ is the only index such
that $\tilde n(k-1) < n \le \tilde n(k)$. By construction $\bar z^{(k)} \to \bar z$,
hence also $z_n \to \bar z$ and it follows by assumption that
$h_n(z_n) \to h(\bar z)$. However, this is impossible because the subsequence
$\{h_{\tilde n(k)}(z_{\tilde n(k)})\}_{k\in\N}$
does not converge to $h(\bar z)$,
as $|h_{\tilde n(k)}(z_{\tilde n(k)}) - h(\bar z)| = |h_{\tilde n(k)}(\bar z^{(k)}) - h(\bar z)|
> \epsilon$ for every $k\in\N$ by \eqref{eq:toolest}.
\end{proof}

\smallskip

\begin{proof}[Proof of Lemma~\ref{th:CC}]
Let us set for $n\in\N$ and $z \in [0,\infty)$
\begin{equation*}
	h_n(z) \,:=\, \frac{n \, a_n}{
	\widehat V^-(0) \, V^+(\lfloor a_n z \rfloor)} \, q_n^+(0, \lfloor a_n z \rfloor) \,.
\end{equation*}
Observe that, if $z > 0$ and $z_n \to z$, then as $n\to\infty$
\begin{equation*}
	V^+(\lfloor a_n z_n \rfloor) \,\sim\, z^{\alpha \rho} \, V^+(a_n) \,\sim\,
	z^{\alpha \rho} \, \frac{\underline{\widehat V}^+(a_n)}{(1-\zeta)} \,\sim\,
	z^{\alpha \rho} \, \frac{C^+ n \, \bbP(\tau_1^- > n)}{(1-\zeta)} \,\sim\,
	z^{\alpha \rho} \, C^+ n \, \bbP(\widehat\tau_1^- > n) \,,
\end{equation*}
having applied the third relation in \eqref{eq:sumup1},
\eqref{eq:VhatV} and the second relation in \eqref{eq:leftco},
the first relation in \eqref{eq:lien}, and \eqref{eq:usefuly0}.
It follows that the sequence of real functions $\{h_n(\cdot)\}_{n\in\N}$,
all defined on $I = [0,\infty)$, satisfies the assumptions of Lemma~\ref{th:remarkable}:
in fact, by the first relations in \eqref{eq:xqlarge} and \eqref{eq:xqsmall},
for every $z \in [0,\infty)$ and every sequence $z_n \to z$ we have that
\begin{equation*}
	\exists h(z) \,:=\, \lim_{n\to\infty} h_n(z_n) \,=\,
	\begin{cases}
	\displaystyle \frac{g^+(z)}{\mathtt{C}^+ \, z^{\alpha\rho}} & \text{ if } z > 0 \\
	\rule{0pt}{1.4em}g(0) & \text{ if } z = 0
	\end{cases} \,.
\end{equation*}
By Lemma~\ref{th:remarkable}, the function $h(\cdot)$ is continuous, hence
$\lim_{z \downarrow 0} h(z) = h(0)$.
Recalling \eqref{eq:DonSav}, it follows that
${\mathtt C}^+ = \tilde {\mathtt C}^+$. With almost identical arguments
one shows that ${\mathtt C}^- = \tilde {\mathtt C}^-$.
\end{proof}


\section{Local limit theorems in the absolutely continuous case}

\label{sec:lltabs}

In this section we focus on the absolutely continuous case,
cf. Hypothesis~\ref{hyp:main2}.
Since the law of $S_1$ has no atom, the distinctions between the different
renewal functions evaporate: $V^+ = \widehat V^+ = \widehat {\underline V}^+
= {\underline V}^+$, and analogously for $V^-$.
Therefore everything will be expressed as a function of $V^+$
and $V^-$.

Our goal is to derive local asymptotic relations for the kernel
$f_N^+(x,y)$, recall \eqref{eq:abscontden+}, that are closely
analogous to the relations stated in Proposition~\ref{th:xq} for the lattice case.
More precisely, we are going to prove the following important result:

\begin{theorem}\label{th:xqabscont}
Assume that Hypothesis~\ref{hyp:main} and Hypothesis~\ref{hyp:main2}
(absolutely continuous case) hold.
Then the following relations hold as $n\to\infty$, for $x,y \in [0,\infty)$:
\begin{gather}
	\label{eq:asi1loc}
	f_n^+(x,y) \,=\,
	\frac{\bbP(\tau_1^- > n)}{a_n} V^-(x)
	\left( g^+\left(\frac{y}{a_n}\right) + o(1) \right) , \ \ \,
	\text{unif. for $x = o(a_n)$, $y \in [0,\infty)$} \\
	\label{eq:asi0loc}
	f_n^+(x, y)
	\;=\; \frac{g(0)}{n a_n} V^-(x)
	V^+(y) 	\left( 1 + o(1) \right) \,,  \ \ \,
	\text{unif. for $x = o(a_n)$, $y = o(a_n)$} \,.
\end{gather}
%
\end{theorem}

It is convenient to introduce the measure
\begin{equation*}
	F_n^+(x, \dd y) := \bbP_x(S_1 > 0, \ldots, S_n > 0, S_n \in \dd y ) \,,
\end{equation*}
so that, cf. \eqref{eq:abscontden+},
\begin{equation*}
	f_n^+(x, y) = \frac{F_n^+(x, \dd y)}{\dd y} \,.
\end{equation*}
Our starting point is the ``Stone version'' of the lattice
estimates in Proposition~\ref{th:xq}, proved by
Vatutin and Wachtel~\cite{cf:VatWac} and Doney~\cite{cf:Don11}, that
read as
\begin{enumerate}
\item for any fixed $\Delta > 0$,
uniformly for $x = o(a_n)$, $y \in [0,\infty)$ one has, as $n\to\infty$,
\begin{equation} \label{eq:asi1}
	F_n^+(x, [y, y+\Delta) )
	\;=\; \frac{\bbP(\tau_1^- > n)}{a_n}
	V^-(x)
	\left( g^+\left(\frac{y}{a_n}\right) \Delta + o(1) \right) \,,
\end{equation}

\item for any fixed $\Delta > 0$,
uniformly for $x = o(a_n)$, $y = o(a_n)$
one has, as $n\to\infty$,
\begin{equation} \label{eq:asi0}
	F_n^+(x, [y, y+\Delta) )
	\;=\; \frac{g(0)}{n a_n}
	V^-(x)
	\bigg( \int_{[y, y+\Delta)}
	V^+(z)
	\, \dd z \bigg)
	\left( 1 + o(1) \right) \,.
\end{equation}
\end{enumerate}
The idea of the proof of Theorem~\ref{th:xqabscont} is to derive
the asymptotic relations for $f_n^+(x,y)$ from the ``integrated'' relations
\eqref{eq:asi1}, \eqref{eq:asi0}, by letting $\Delta \downarrow 0$. The delicate
point is that interchanging the limits $\Delta \downarrow 0$
and $n\to\infty$ requires a careful justification.


\subsection{Strategy of the proof}

We choose ${\bar k}\in\N$ sufficiently large
but \emph{fixed}, as we specify below
(cf. \S\ref{sec:choicek}), and we write for all $n\ge {\bar k}$
\begin{equation} \label{eq:origine}
	f_n^+(x,y) \;=\; \int_{[0,\infty)} \dd z \,
	F_{n-{\bar k}}^+(x, \dd z) \, f_{\bar k}^+(z, y) \,.
\end{equation}
Next we approximate this integral by a Riemann sum over small
intervals. More precisely, we set for $z,y \ge 0$ and $\Delta > 0$
\begin{equation} \label{eq:cC+}
	\underline{f}_\Delta^+(z,y) := \inf_{u \in [z, z+\Delta)} f_{\bar k}^+(u,y) \,, \qquad
	\overline{f}_\Delta^+(z,y) := \sup_{u \in [z, z+\Delta)} f_{\bar k}^+(u,y) \,,
\end{equation}
so that for every $\Delta > 0$, $n\in\N$ and $x,y \ge 0$ we can write
\begin{equation} \label{eq:comparison}
	s_{n,\Delta}^+(x,y) \,\le\,
	f_n^+(x,y) \,\le\,
	S_{n,\Delta}^+(x,y) \,,
\end{equation}
where
\begin{gather}
	\label{eq:sRi}
	s_{n,\Delta}^+(x,y) \,:=\,
	\sum_{z \in \Delta\N_0}
	F^+_{n-{\bar k}}(x, [z,z+\Delta)) \, \underline{f}_\Delta^+(z,y) \,, \\
	S_{n,\Delta}^+(x,y) \,:=\,
	\label{eq:SRi}
	\sum_{z \in \Delta\N_0}
	F^+_{n-{\bar k}}(x, [z,z+\Delta)) \, \overline{f}_\Delta^+(z,y) \,.
\end{gather}
The idea is to replace $F_{n-{\bar k}}^+(x,[z, z+\Delta))$
by its asymptotic behavior,
given in \eqref{eq:asi1} and \eqref{eq:asi0},
and to show that $\underline{f}_\Delta^+(z,y) \simeq \overline{f}_\Delta^+(z,y)$ if $\Delta$ is small.
This is of course to be made precise. The delicate point is that we need
uniformity in $z$.

\subsection{The choice of ${\bar k}$}
\label{sec:choicek}

For the choice of ${\bar k}$ appearing in \eqref{eq:origine}
we impose two conditions.

The first condition on $\bar k$ is that $f_{{\bar k}-1}$ is a bounded function,
which we can do by Hypothesis~\ref{hyp:main2}. This is
enough to ensure that $f_{\bar k}^+(z,y)$ is uniformly continuous in $z$, uniformly
in $y$. By this we mean that for every $\epsilon > 0$ there exists $\Delta > 0$ such that
for all $z, z' \ge 0$ with $|z-z'|\le \Delta$ and for every $y \ge 0$
\begin{equation*}
	| f_{\bar k}^+(z',y) - f_{\bar k}^+(z,y) | \le \epsilon \,.
\end{equation*}
The proof is simple:
\begin{equation*}
\begin{split}
	| f_{\bar k}^+(z',y) & - f_{\bar k}^+(z,y) | = \bigg| \int_{w \in [0,\infty)} (f(w-z') - f(w - z))
	f_{{\bar k}-1}^+(w, y) \, \dd w \bigg| \\
	& \le \|f_{{\bar k}-1}\|_\infty \bigg| \int_{w \in [0,\infty)} (f(w-z') - f(w - z)) \, \dd w \bigg| \\
	& \le \|f_{{\bar k}-1}\|_\infty \bigg| \int_{w \in \R} (f(w + (z'-z)) - f(w)) \, \dd w \bigg|
	= \|f_{{\bar k}-1}\|_\infty \|\Theta_{(z'-z)}f - f\|_{L^1} \,,
\end{split}
\end{equation*}
where $(\Theta_h f)(x) := f(x - h)$ denotes the translation operator. Since this is continuous
in $L^1$, the claim follows.

Let us now set, for ${\bar k}\in\N$, $\alpha' > 0$, $\Delta > 0$ and $z \in \R$,
\begin{equation} \label{eq:Corig}
	\tilde C^{\alpha'}_\Delta(z) := \sup_{u \in [z, z+\Delta)} (1+|u|)^{\alpha'} \, f_{\bar k}(u) \,.
\end{equation}
The second condition on ${\bar k}$ is that (it is large enough so that)
for some $\alpha' \in (\rho\alpha,\alpha)$
and for some (hence any) $\Delta > 0$ the following upper Riemann sum is finite:
\begin{equation} \label{eq:dRihyp}
	\Theta(\Delta) \,:=\,
	\sum_{w \in \Delta\Z} \Delta \cdot \tilde C^{\alpha'}_\Delta(w) \,<\, \infty \,.
\end{equation}
In other words, we require that the function $(1+|w|)^{\alpha '} \, f_{\bar k}(w)$ is
\emph{directly Riemann integrable}, cf. \cite[\S XI.1]{cf:Fel2}.
We point out that this condition is always satisfied if ${\bar k}$ is large enough,
with no further assumptions beyond Hypothese~\ref{hyp:main} and~\ref{hyp:main2},
as it is proved in \cite{cf:Car2}.
Of course, an immediate sufficient condition, very common in concrete applications, is that
there exists $\alpha '' \in (\rho\alpha, \alpha)$ such that
$f_{\bar k}(x) \le (const.) / |x|^{1 + \alpha''}$.

A direct consequence of \eqref{eq:dRihyp} is that, for any fixed $\Delta > 0$,
the contribution to the sum of the terms with $|w| > M$ is small,
provided $M$ is large. The interesting point is that \emph{$M$ can be chosen
independently of (bounded) $\Delta$}: more precisely,
\begin{equation} \label{eq:dRicons}
	\forall \Delta_0 > 0, \ \forall \eta > 0 \quad \exists M > 0: \quad
	\sumtwo{w \in \Delta\Z}{|w| > M} \Delta \cdot
	\tilde C^{\alpha'}_\Delta(w) \,<\, \eta \,,
	\quad \forall \Delta \in (0,\Delta_0) \,.
\end{equation}
This follows by the \emph{monotonicity properties of upper Riemann
sums}. In fact, recalling the definition of $\Theta(\Delta)$ in \eqref{eq:dRihyp},
by construction one has $\Theta(\frac{1}{2}\Delta) \le \Theta(\Delta)$ for every $\Delta$
and $\Theta(\Delta) \le 2 \Theta(\Delta')$ for $\Delta \in (\frac{1}{2}\Delta', \Delta')$.
By iteration, it then suffices to prove \eqref{eq:dRicons} for fixed $\Delta$
(which, as we just remarked, follows immediately from
the finiteness of the sum in \eqref{eq:dRihyp}).

\subsection{Some preliminary results}

A useful observation is that the function $V^+$ is
increasing and sub-additive (as every renewal function), hence
\begin{equation} \label{eq:smarag}
	\frac{V^+(y+\delta)}{V^+(y)} \,\le\, 1 + \frac{V^+(\delta)}{V^+(y)}
	\,\le\, 1 + \frac{V^+(\delta)}{V^+(0)} \,=\, 1 + V^+(\delta) \,, \qquad
	\forall \delta, y \ge 0 \,.
\end{equation}

Let us first derive an upper bound from \eqref{eq:asi1} and \eqref{eq:asi0}.

\begin{lemma}
Fix any sequence $x_n = o(a_n)$ and $\Delta > 0$.
There exists a constant $C \in (0,\infty)$ such that
\begin{equation}\label{eq:ubb1}
	F_n^+(x_n,[y,y+\Delta)) \,\le\, \frac{C \, \Delta}{n a_n} \, V^-(x_n) \, V^+(y) \,,
	\quad \ \ \forall n\in\N, \ \forall y \ge 0 \,.
\end{equation}
\end{lemma}

\begin{proof}
We proceed by contradiciton. If \eqref{eq:ubb1} doesn't hold, there exist
$\Delta > 0$ and sequences
$C_n \to +\infty$, $y_n \ge 0$ such that
along a subsequence $n = n_k \to +\infty$ we have
\begin{equation} \label{eq:contrad}
	F_n^+(x_n,[y_n,y_n+\Delta)) \,>\, \frac{C_n \, \Delta}{n a_n} \, V^-(x_n)
	\, V^+(y_n) \,.
\end{equation}
For ease of notation, we implicitly assume that $n = n_k$ until the end of the proof.
Extracting a suitable subsequence, we may assume that
$y_{n} / a_{n} \to c \in [0,+\infty]$ and we show that in each case
$c=0$, $c \in (0,\infty)$ and $c = +\infty$ we obtain a contradiction.
\begin{itemize}
\item If $c=0$ then $y_n = o(a_n)$ and \eqref{eq:contrad} contradicts \eqref{eq:asi0},
because, by \eqref{eq:smarag},
\begin{equation*}
	\int_{[y_n, y_n+\Delta)} V^+(z) \, \dd z \,\le\, \Delta \, V^+(y_n+\Delta)
	\,\le\, (const.) \, V^+(y_n) \,.
\end{equation*}

\item If $0 < c < \infty$ then $y_n \sim  c \, a_n$ and
$V^+(y_n) \sim c^{\alpha\rho} V^+(a_n)$,
because $V^+ \in R_{\alpha\rho}$. By \eqref{eq:lien}
we know that $V^+(a_n) \sim (const.) \, n
\bbP(\tau_1^- > n)$, hence from \eqref{eq:contrad} we get
\begin{equation} \label{eq:contrad2}
	F_n^+(x_n,[y_n,y_n+\Delta)) \,>\, (const.') \frac{C_n \, \Delta}{a_n} \, V^-(x_n)
	\, \bbP(\tau_1^- > n) \,.
\end{equation}
This is in contradiction with \eqref{eq:asi1},
because $g^+$ is bounded.

\item The case $c=\infty$ is analogous and even simpler. In fact in this case
$V^+(y_n) \ge V^+(a_n)$ for large $n$ (recall that $V^+$ is increasing)
hence we still have \eqref{eq:contrad2}, which is again
in contradiction with \eqref{eq:asi1} (which holds also for $y_n \gg a_n$).
\end{itemize}
This completes the proof.
\end{proof}

Then we prove a crucial approximation result: we show that
the sum in \eqref{eq:SRi} can be truncated
to values of $z$ at a finite distance from $y$, losing a negligible contribution.

\begin{lemma} \label{th:trunc}
Fix any sequence $x_n = o(a_n)$ and $\Delta_0 > 0$.
For every $\eta > 0$ there exists $M \in (0,\infty)$ such that
for all $n\in\N$, $\Delta \in (0,\Delta_0)$ and $y\ge 0$
\begin{equation*}
	\sumtwo{z \in \Delta\N_0}{|z - y| > M}
	F^+_{n-{\bar k}}(x_n, [z,z+\Delta)) \, \overline{f}_\Delta^+(z,y) \,\le\,
	\eta \bigg( \frac{g(0)}{n a_n} V^-(x_n) V^+(y) \bigg) \,.
\end{equation*}
\end{lemma}

\begin{proof}
A first basic observation is that it is enough to prove this relation \emph{for
a fixed $\Delta > 0$},
thanks to the monotonicity properties of upper Riemann sums
(cf. the lines following \eqref{eq:dRicons}).
For any $M>0$, by \eqref{eq:ubb1} we can write
\begin{equation*}
	\sumtwo{z \in \Delta\N_0}{|z - y| > M}
	F^+_{n-{\bar k}}(x_n, [z,z+\Delta)) \, \overline{f}_\Delta^+(z,y)
	\,\le\, \frac{C}{n a_n} \, V^-(x_n)
	\sumtwo{z \in \Delta\N_0}{|z - y| > M}  \Delta \cdot
	V^+(z) \, \overline{f}_\Delta^+(z,y) \,.
\end{equation*}
Since $V^+$ is sub-additive and increasing,
for $z > y$ we have $V^+(z) \le V^+(y) + V^+(z-y)$,
while for $z < y$ we have
$V^+(z) \le V^+(y)$.
Therefore we can bound
\begin{equation} \label{eq:treee}
\begin{split}
	& \sumtwo{z \in \Delta\N_0}{|z - y| > M}
	\Delta \cdot V^+(z) \, \overline{f}_\Delta^+(z,y) \\
	& \qquad \,\le\,
	V^+(y) \sumtwo{z \in \Delta\N_0}{z < y - M}
	\Delta \cdot \overline{f}_\Delta^+(z,y) \,+\,
	\sumtwo{z \in \Delta\N_0}{z > y + M}
	\big( V^+(y) + V^+(z-y) \big) \Delta \cdot \overline{f}_\Delta^+(z,y) \\
	& \qquad \,=\,
	V^+(y) \sumtwo{z \in \Delta\N_0}{|z-y| > M}
	\Delta \cdot \overline{f}_\Delta^+(z,y) \,+\,  \sumtwo{z \in \Delta\N_0}{z > y + M}
	V^+(z-y) \,\Delta \cdot \overline{f}_\Delta^+(z,y) \,.
\end{split}
\end{equation}
Let us bound the second sum.
Recalling \eqref{eq:cC+} and observe that
$f_{\bar k}^+(u,y) \le f_{\bar k}(y-u)$ for all $u,y \ge 0$; moreover, for
any fixed $\alpha' \in (\rho\alpha,\alpha)$ one has $V^+(w) \le |w|^{\alpha '}$
for large $w$, because $V^+ \in R_{\alpha\rho}$.
Recalling \eqref{eq:Corig},
it follows that $V^+(z-y) \, \overline{f}_\Delta^+(z,y)
\le (const.) \, \tilde C_\Delta^{\alpha '}(y-z)$.
Coming back to \eqref{eq:treee},
we use the bound $ \overline{f}_\Delta^+(z,y) \le \tilde
C_\Delta^{\alpha '}(y-z)$ in the first sum of the last line.
Since $1 = V^+(0) \le V^+(y)$,
it follows that for some constant $c > 0$
\begin{equation*}
\begin{split}
	& \sumtwo{z \in \Delta\N_0}{|z - y| > M} \Delta \cdot
	V^+(z) \, \overline{f}_\Delta^+(z,y)
	\,\le\, c \, V^+(y) \sumtwo{w \in \Delta\Z}{|w| > M}
	\Delta \cdot \tilde C^{\alpha'}_\Delta(w) \,.
\end{split}
\end{equation*}
We now apply \eqref{eq:dRicons} with $\eta$ replaced by $\eta g(0)/(cC)$,
getting
\begin{equation*}
	\sumtwo{z \in \Delta\N_0}{|z - y| > M}
	F^+_{n-{\bar k}}(x_n, [z,z+\Delta)) \, \overline{f}_\Delta^+(z,y)
	\,\le\, \eta \, \frac{g(0)}{n a_n} \, V^-(x_n)
	V^+(y+\Delta) \,,
\end{equation*}
which is precisely what we want to prove (recall \eqref{eq:SRi}).
\end{proof}

\subsection{Proof of \eqref{eq:asi0loc}}

We can finally prove \eqref{eq:asi0loc}. From the discussion of section \ref{sec:notation}, 
it suffices to show that, if we \emph{fix} any two sequences
$x_n = o(a_n)$ and $y_n = o(a_n)$, for every $\epsilon > 0$ we have, for large $n$,
\begin{equation}\label{eq:aimmi}
	(1-\epsilon) V^+(y_n) \,\le\, \frac{f_n^+(x_n, y_n)}
	{\frac{g(0)}{n a_n} V^-(x_n)} \,\le\, (1+\epsilon) V^+(y_n) \,.
\end{equation}

Let $\epsilon \in (0,1)$ be fixed. We start choosing $\Delta_0 := 1$
and we let $M_0$ denote the constant $M$ in Lemma~\ref{th:trunc}
corresponding to $\eta = \epsilon/4$. The reason for this will be clear later.

Observe that, by time-reversal, relation \eqref{eq:Valt}
for $N = \bar k$ can be rewritten as
\begin{equation} \label{eq:aimmi34}
	\int_{[0,\infty)} V^+(u) \, f_{\bar k}^+(u,y) \, \dd u  \,=\, V^+(y) \,, \qquad
	\forall y \ge 0\,.
\end{equation}
For notational convenience, let us set $V^+(u) := 0$ for $u < 0$, so that
the domain of integration can be extended to $\R$.
We claim that there exists ${\bar M} > 0$ such that
\begin{equation} \label{eq:aimmi4}
	\int_{|u - y| > {\bar M}-1} V^+(u) \, f_{\bar k}^+(u,y) \, \dd u  \,\le\,
	\frac{\epsilon}{4} V^+(y) \,, \qquad \forall y \ge 0 \,.
\end{equation}
In fact, observe that $f_{n}^+(x,y) \le f_{n}(y - x)$ and that
for every $y \ge 0$ and $s\in \R$
\begin{equation*}
	V^+(y-s) \,\le\, V^+(y) + V^+(|s|) \,\le\, 2 \, V^+(y)V^+(|s|) \,,
\end{equation*}
by sub-additivity of $V^+(\cdot)$
and the fact that $1 = V^+(0) \le V^+(x)$ for every $x\ge 0$, hence
\begin{equation*}
\begin{split}
	& \int_{|u - y| > M-1} V^+(u) \, f_{\bar k}^+(u,y) \, \dd u  \,\le\,
	\int_{|u - y| > M-1} V^+(u) \, f_{\bar k}(y-u) \, \dd u  \\
	& \qquad \qquad \,=\, \bbE (V^+(y-S_{\bar k}) \ind_{\{|S_{\bar k}| > M-1\}})
	\,\le\, 2 \, V^+(y) \, \bbE(V^+(|S_{\bar k}|) \, \ind_{\{|S_{\bar k}| > M-1\}}) \,.
\end{split}
\end{equation*}
We have $\bbP(|S_{\bar k}| > \cdot) \le \bar{k} \, \bbP(|S_{1}| > \cdot) \in R_{-\alpha}$
(recall Hypothesis~\ref{hyp:main}), hence
$\bbE(|S_{\bar k}|^{\alpha '}) < \infty$ for all $\alpha' < \alpha$. Since
$V^+(\cdot) \in R_{\alpha\rho}$ by \eqref{eq:sumup1} and since $\rho < 1$,
it follows that $\bbE(V^+(|S_{\bar k}|)) < \infty$.
We can then choose $\bar M>0$ large enough, with
$\bar M > M_0$ (that was fixed above), so that
$\bbE(V^+(|S_{\bar k}|) \, \ind_{\{|S_{\bar k}| > {\bar M}-1\}}) \le \epsilon/8$.
Relation \eqref{eq:aimmi4} is proved.

It follows immediately from
\eqref{eq:aimmi34}, \eqref{eq:aimmi4} that for every $0 < \Delta \le \Delta_0$
and $y\ge 0$ one has
\begin{equation}\label{eq:leylaan}
	\Bigg| \sumtwo{z \in \Delta\N_0}{|z - y| \le {\bar M}}
	\bigg( \int_{[z, z+\Delta)} V^+(u) \, f_{\bar k}^+(u,y) \, \dd u \bigg)
	\;-\; V^+(y) \Bigg|
	\,\le\, \frac{\epsilon}{4} \, V^+(y)  \,,
\end{equation}
just because the intervals $[z, z+\Delta)$, as $z$ varies in $\Delta\N_0$
with $|z - y| \le {\bar M}$, are disjoint and their union contains $[y-(\bar M-1), y+(\bar M-1))$.

Next observe that, for any $\Delta \le \Delta_0 = 1$
and $y \ge 0$, by \eqref{eq:smarag}
\begin{equation*}
\begin{split}
	& \sumtwo{z \in \Delta\N_0}{|z - y| \le {\bar M}}
	\int_{[z, z+\Delta)} V^+(u) \, \dd u \,\le\,
	\int_{y - {\bar M}}^{y + {\bar M} + \Delta} V^+(u) \, \dd u
	\,\le\, (2{\bar M}+\Delta) \, V^+(y + {\bar M} + \Delta) \\
	& \qquad
	\,\le\, (2{\bar M}+1) \, (1 + V^+({\bar M}+1)) \, V^+(y)
	\,=:\, {\bar C}\, V^+(y) \,,
\end{split}
\end{equation*}
where we stress that ${\bar C} > 0$ is a constant depending only on $\epsilon$
(through $\bar M$). Recall the definition \eqref{eq:cC+}.
Since we have chosen ${\bar k}$ large enough so that $f_{\bar k}^+(u,y)$ is uniformly
continuous in $u$, uniformly in $y$,
we can choose $0 < \bar\Delta < \Delta_0$ small enough so that
\begin{equation*}
	\forall z, y \ge 0, \ \forall u \in [z,z+{\bar\Delta}): \quad \
	\overline{f}_{\bar\Delta}^+(z,y) - \frac{\epsilon}{4 \bar C} \,\le\, f_{\bar k}^+(u,y) \,\le\,
	\underline{f}_{\bar\Delta}^+(z,y) + \frac{\epsilon}{4 \bar C} \,.
\end{equation*}
Inserting these estimates in \eqref{eq:leylaan}, it follows that, for every $y\ge 0$,
\begin{gather}
	\label{eq:aimmi2a}
	\sumtwo{z \in {\bar\Delta}\N_0}{|z - y| \le {\bar M}}
	\bigg( \int_{[z, z+{\bar\Delta})} V^+(u) \, \dd u \bigg)
	\, \overline{f}_{\bar\Delta}^+(z,y) \, \le \, \bigg( 1 + \frac{\epsilon}{2} \bigg) V^+(y) \\
	\label{eq:aimmi2b}
	\sumtwo{z \in {\bar\Delta}\N_0}{|z - y| \le {\bar M}}
	\bigg( \int_{[z, z+{\bar\Delta})} V^+(u) \, \dd u \bigg)
	\, \underline{f}_{\bar\Delta}^+(z,y) \, \ge \, \bigg( 1 - \frac{\epsilon}{2} \bigg) V^+(y) \,.
\end{gather}

Fix two sequences $x_n, y_n = o(a_n)$ and consider
the previous relations with $y = y_n$. Note that the sum over $z$ ranges over
a finite number of points, all at finite distance from $y_n$, hence each $z$ in the
sum is $o(a_n)$. Then it follows from \eqref{eq:asi0} that there exists $n_0 = n_0(\epsilon) < \infty$
such that for all $n\ge n_0$ and for all
$z \in {\bar\Delta}\N_0$ with $|z - y_n| \le {\bar M}$
\begin{equation*}
	\frac{1+\frac{\epsilon}{2}}{1+\frac{3\epsilon}{4}} \,
	\frac{F^+_{n-{\bar k}}(x_n, [z,z+{\bar\Delta}))}
	{\frac{g(0)}{n a_n} V^-(x_n)} \,\le\,
	\int_{[z, z+{\bar\Delta})} V^+(u) \, \dd u \,\le\,
	\frac{1-\frac{\epsilon}{2}}{1-\epsilon} \,
	\frac{F^+_{n-{\bar k}} (x_n, [z,z+{\bar\Delta}))}{\frac{g(0)}{n a_n} V^-(x_n)} \,.
\end{equation*}
Relations \eqref{eq:aimmi2a}, \eqref{eq:aimmi2b} then yield
\begin{gather}
	\label{eq:aimmi1a}
	\sumtwo{z \in {\bar\Delta}\N_0}{|z - y_n| \le {\bar M}}
	\frac{F^+_{n-{\bar k}}(x_n, [z,z+{\bar\Delta}))}
	{\frac{g(0)}{n a_n} V^-(x_n)} \, \overline{f}_{\bar\Delta}^+(z,y_n)
	\, \le \, \bigg(1+\frac{3\epsilon}{4}\bigg) V^+(y_n) \\
	\label{eq:aimmi1b}
	\sumtwo{z \in {\bar\Delta}\N_0}{|z - y_n| \le {\bar M}}
	\frac{F^+_{n-{\bar k}}(x_n, [z,z+{\bar\Delta}))}
	{\frac{g(0)}{n a_n} V^-(x_n)}
	\, \underline{f}_{\bar\Delta}^+(z,y_n) \, \ge \, (1-\epsilon) V^+(y_n) \,.
\end{gather}
Recall that our choice of $\bar M \ge M_0$ was such that Lemma~\ref{th:trunc} holds
for $\eta = \epsilon/4$. Therefore we can drop the restriction $|z - y_n| \le {\bar M}$
in \eqref{eq:aimmi1a}, provided we replace $\frac{3\epsilon}{4}$
by $\epsilon$ in the right hand side. Plainly, the restriction $|z - y_n| \le {\bar M}$
can be dropped from the sum in \eqref{eq:aimmi1b} with no further modification.
Looking back at \eqref{eq:comparison}--\eqref{eq:SRi}, it follows
that \eqref{eq:aimmi} holds true for $n\ge n_0$.
This completes the proof of relation~\eqref{eq:asi0loc}.

\subsection{Proof of \eqref{eq:asi1loc}}

The proof of \eqref{eq:asi1loc} is close in spirit to that of \eqref{eq:asi0loc} just given.
It suffices to show that, if we \emph{fix} any two sequences
$x_n = o(a_n)$ and $y_n \ge 0$, for every $\epsilon > 0$ we have, for large $n$,
\begin{equation}\label{eq:aimmi-bis}
	g^+ \bigg(\frac{y_n}{a_n} \bigg) \,-\, \epsilon \,\le\, \frac{f_n^+(x_n, y_n)}
	{\frac{1}{a_n} \bbP(\tau_1^- > n) V^-(x_n)} \,\le\,
	g^+ \bigg(\frac{y_n}{a_n} \bigg) \,+\, \epsilon \,.
\end{equation}
As we remarked in \S\ref{sec:notation}, it suffices to consider sequences
that have a (possibly infinite) limit, so we assume that $y_n / a_n \to \kappa \in [0,+\infty]$.
The case $\kappa=0$, i.e. $y_n = o(a_n)$, is a consequence of
relation \eqref{eq:asi0loc}, which is a stronger statement, so there is nothing to prove.
We then treat separately the cases $\kappa \in (0,\infty)$
and $\kappa = \infty$, starting from the former.

\smallskip

Let $\epsilon \in (0,1)$ be fixed. We choose $\Delta_0 := 1$
and we let $M_0$ denote the constant $M$ in Lemma~\ref{th:trunc}
corresponding to $\eta = \epsilon/(8 g(0) \mathtt{C}^+ \kappa^{\alpha \rho})$
(the reason for this choice will be clear later),
where we recall that $\mathtt{C}^+$ is the
constant appearing in~\eqref{eq:lien}.

Recall that $y_n \sim \kappa a_n$ with $\kappa \in (0,\infty)$
and note that $g^+(\kappa) > 0$. We claim that
we can choose $n_0 \in \N_0$ and $\bar M > 0$ such that
\begin{equation} \label{eq:aimmi4-bis}
	1 - \frac{\epsilon}{4 \, g^+(\kappa)} \,\le\,
	\int_{|u - y_n| \le {\bar M}-1} f_{\bar k}^+(u,y_n) \, \dd u  \,\le\,
	1 \,, \qquad \forall n \ge n_0 \,.
\end{equation}
In fact, if we denote by $S^* := -S$ the reflected walk,
if $y_n \ge M - 1$ we can write
\begin{equation*}
\begin{split}
	&\int_{|u - y_n| \le M-1} f_{\bar k}^+(u,y_n) \, \dd u  \\
	& \qquad \,=\,
	\bbP_{y_n} (S^*_1 \ge 0, \ldots, S^*_{\bar k - 1} \ge 0, S^*_{\bar k}
	\in [y_n - (M - 1), y_n + (M - 1)] ) \\
	& \qquad \,=\, \bbP_{0} (S^*_1 \ge -y_n, \ldots, S^*_{\bar k - 1} \ge -y_n, S^*_{\bar k}
	\in [- (M - 1), + (M - 1)] ) \,,
\end{split}
\end{equation*}
from which the upper bound in \eqref{eq:aimmi4-bis} follows trivially. For the lower bound,
note that
\begin{equation*}
\begin{split}
	&1 - \int_{|u - y_n| \le M-1} f_{\bar k}^+(u,y_n) \, \dd u  \\
	& \qquad \,\le\, \bbP (\{S^*_1 \ge -y_n, \ldots, S^*_{\bar k - 1} \ge -y_n\}^c)
	\,+\, \bbP(S^*_{\bar k} \not\in [- (M - 1), + (M - 1)] ) \,.
\end{split}
\end{equation*}
Since $\bar k$ is fixed and $y_n \to +\infty$,
the first term in the right hand side vanishes as $n\to\infty$, hence
we can choose $n_0$ such that for $n\ge n_0$ it is less than
$\epsilon/(8 \, g^+(\kappa))$. Analogously, we choose $\bar M$
large enough, with $\bar M > M_0$ (that was fixed above),
such that for $M \ge \bar M$ the second term in the right hand side
is less than $\epsilon/(8 \, g^+(\kappa))$.
Equation \eqref{eq:aimmi4-bis} is proved.

It follows immediately from
\eqref{eq:aimmi4-bis} that for every $0 < \Delta \le \Delta_0$
and $n \ge n_0$ one has
\begin{equation}\label{eq:leylaan-bis}
	1 - \frac{\epsilon}{4 \, g^+(\kappa)} \,\le\,
	\sumtwo{z \in \Delta\N_0}{|z - y_n| \le {\bar M}}
	\bigg( \int_{[z, z+\Delta)} \, f_{\bar k}^+(u,y_n) \, \dd u \bigg)
	\,\le\, 1 \,,
\end{equation}
just because the intervals $[z, z+\Delta)$, as $z$ varies in $\Delta\N_0$
with $|z - y_n| \le {\bar M}$, are disjoint and their union contains $[y_n-(\bar M-1), y_n+(\bar M-1))$.

Recall the definition \eqref{eq:cC+}.
Since we have chosen ${\bar k}$ large enough so that $f_{\bar k}^+(u,y)$ is uniformly
continuous in $u$, uniformly in $y$,
we can choose $0 < \bar\Delta < \Delta_0$ small enough so that for all $n \ge n_0$,
$z \ge 0$ and for every $u \in [z,z+{\bar\Delta})$
\begin{equation*}
	\overline{f}_{\bar\Delta}^+(z,y_n) - \frac{\epsilon}{4 \, (2\bar M+1) \, g^+(\kappa)}
	\,\le\, f_{\bar k}^+(u,y_n) \,\le\,
	\underline{f}_{\bar\Delta}^+(z,y_n) + \frac{\epsilon}{4 \, (2\bar M+1) \, g^+(\kappa)} \,.
\end{equation*}
Plugging these estimates into \eqref{eq:leylaan-bis} and observing that there are
at most $(2 \bar M + 1)/\bar\Delta$ values of $z \in \bar\Delta \N_0$ such that
$|z-y_n| \le \bar M$, we obtain
\begin{gather}
	\label{eq:aimmi2a-bis}
	\sumtwo{z \in {\bar\Delta}\N_0}{|z - y_n| \le {\bar M}}
	\Delta \, \overline{f}_{\bar\Delta}^+(z,y_n)
	\, \le \, 1 + \frac{\epsilon}{4 \, g^+(\kappa)} \,, \qquad
	\sumtwo{z \in {\bar\Delta}\N_0}{|z - y_n| \le {\bar M}}
	\Delta\,
	\, \underline{f}_{\bar\Delta}^+(z,y_n) \, \ge \,
	1 - \frac{\epsilon}{2 \, g^+(\kappa)} \,.
\end{gather}

Observe that the sum in both the preceding relations ranges over
a finite number of $z$, all at finite distance from $y_n$, hence each $z$ in the
sum is such that $z/a_n \to \kappa \in (0,\infty)$.
Therefore $g^+(z/a_n) \to g(\kappa)$ as $n\to\infty$, uniformly over $z$ in the sum range.
It follows that there exists $n_1 \ge n_0$ such that for all $n \ge n_1$
\begin{gather}
	\label{eq:aimmi2a-bis2}
	\sumtwo{z \in {\bar\Delta}\N_0}{|z - y_n| \le {\bar M}}
	g^+\bigg(\frac{z}{a_n}\bigg) \, \Delta \, \overline{f}_{\bar\Delta}^+(z,y_n)
	\, \le \, g^+(\kappa) + \frac{\epsilon}{2} \,,\\
	\label{eq:aimmi2b-bis2}
	\sumtwo{z \in {\bar\Delta}\N_0}{|z - y_n| \le {\bar M}}
	g^+\bigg(\frac{z}{a_n}\bigg) \, \Delta\,
	\, \underline{f}_{\bar\Delta}^+(z,y_n) \, \ge \,
	g^+(\kappa) - \frac{3\epsilon}{4} \,.
\end{gather}
Next observe that, as $n\to\infty$, we have $g^+(y_n/a_n) \to g^+(\kappa)$, and by \eqref{eq:asi1}
\begin{equation*}
	\bigg| g^+\bigg(\frac{z}{a_n}\bigg) \,\Delta \,-\, \frac{F^+_{n-{\bar k}}(x_n, [z,z+{\bar\Delta}))}
	{\frac{1}{a_n} \bbP(\tau_1^- > n) V^-(x_n)} \bigg| \longrightarrow 0 \,,
\end{equation*}
uniformly over $z$ in the sum range of \eqref{eq:aimmi2a-bis2} and \eqref{eq:aimmi2b-bis2}.
It follows that there exists $n_2 \ge n_1$ such that for all $n\ge n_2$
\begin{gather}
	\label{eq:aimmi2a-bis3}
	\sumtwo{z \in {\bar\Delta}\N_0}{|z - y_n| \le {\bar M}}
	\frac{F^+_{n-{\bar k}}(x_n, [z,z+{\bar\Delta}))}
	{\frac{1}{a_n} \bbP(\tau_1^- > n) V^-(x_n)} \, \overline{f}_{\bar\Delta}^+(z,y_n)
	\, \le \, g^+ \bigg(\frac{y_n}{a_n} \bigg) + \frac{3\epsilon}{4} \,,\\
	\label{eq:aimmi2b-bis3}
	\sumtwo{z \in {\bar\Delta}\N_0}{|z - y_n| \le {\bar M}}
	\frac{F^+_{n-{\bar k}}(x_n, [z,z+{\bar\Delta}))}
	{\frac{1}{a_n} \bbP(\tau_1^- > n) V^-(x_n)}
	\, \underline{f}_{\bar\Delta}^+(z,y_n) \, \ge \,
	g^+ \bigg(\frac{y_n}{a_n} \bigg) - \epsilon \,.
\end{gather}
Dropping the restriction $|z - y_n| \le {\bar M}$ in the sum in \eqref{eq:aimmi2b-bis3}
and recalling \eqref{eq:comparison}--\eqref{eq:SRi}, it follows
that the lower bound in \eqref{eq:aimmi-bis} holds true for $n\ge n_2$.

In order to drop the restriction $|z - y_n| \le {\bar M}$
in the sum in \eqref{eq:aimmi2a-bis3} as well, we need to control
the contribution of the terms with $|z - y_n| > {\bar M}$.
Recall that $\bar M$ was chosen greater than $M_0$, in such a way that
Lemma~\ref{th:trunc} holds with $\eta = \epsilon/(8 g(0) \mathtt{C}^+
\kappa^{\alpha\rho})$. Since $y_n \sim \kappa a_n$ and $V^+ \in R_{\alpha\rho}$
by \eqref{eq:sumup1}, it follows that $V^+(y_n) \sim \kappa^{\alpha\rho}
V^+(a_n) \sim \mathtt{C}^+ \kappa^{\alpha\rho} n \bbP(\tau_1^- > n)$,
having applied \eqref{eq:lien}. Therefore there exists $n_3 \ge n_2$ such
that for $n\ge n_3$ one has
$V^+(y_n) \le 2 \mathtt{C}^+ \kappa^{\alpha\rho} n \bbP(\tau_1^- > n)$,
hence by Lemma~\ref{th:trunc}
\begin{equation*}
	\sumtwo{z \in {\bar\Delta}\N_0}{|z - y_n| > {\bar M}}
	\frac{F^+_{n-{\bar k}}(x_n, [z,z+{\bar\Delta}))}
	{\frac{1}{a_n} \bbP(\tau_1^- > n) V^-(x_n)} \, \overline{f}_{\bar\Delta}^+(z,y_n)
	\,\le\, \frac{\epsilon}{4}\,.
\end{equation*}
This means that we can drop the restriction $|z - y_n| \le {\bar M}$ in
\eqref{eq:aimmi2a-bis3},
provided we replace $\frac{3\epsilon}{4}$ by $\epsilon$ in the right hand side.
Recalling \eqref{eq:comparison}--\eqref{eq:SRi}, we have proved that
the upper bound in \eqref{eq:aimmi-bis} holds true for $n\ge n_3$.

Finally, it remains to prove \eqref{eq:aimmi-bis} in the case when
$\kappa = \lim_{n\to\infty} y_n/a_n = +\infty$. Since $g^+(x) \to 0$
as $x \to +\infty$, it suffices to show that for every $\epsilon > 0$, for $n$ large,
\begin{equation} \label{eq:atlaast}
	\frac{f_n^+(x_n, y_n)}
	{\frac{1}{a_n} \bbP(\tau_1^- > n) V^-(x_n)} \,\le\, \epsilon \,,
\end{equation}
where we recall that $x_n = o(a_n)$ is a fixed sequence.
We fix an arbitrary $\Delta$
(say $\Delta = 1$) and note that,
by the upper bound in \eqref{eq:comparison}, we can write
\begin{equation} \label{eq:daidaidai}
	\frac{f_n^+(x_n, y_n)}
	{\frac{1}{a_n} \bbP(\tau_1^- > n) V^-(x_n)} \,\le\,
	\sum_{z \in \Delta\N_0} \frac{F^+_{n-{\bar k}}(x_n, [z,z+1))}
	{\frac{1}{a_n} \bbP(\tau_1^- > n) V^-(x_n)} \, \overline{f}_\Delta^+(z,y_n) \,.
\end{equation}
Since the function $g^+(\cdot)$ is bounded, by \eqref{eq:asi1}
\begin{equation*}
	\hat c \,:=\, \sup_{z \in [0,\infty),\, n > \bar{k}} \,
	\frac{F^+_{n-{\bar k}}(x_n, [z,z+\Delta))}
	{\frac{1}{a_n} \bbP(\tau_1^- > n) V^-(x_n) \Delta} \,<\, \infty \,.
\end{equation*}
Observing that $f_n^+(x,y) \le f_n(y-x)$ and recalling
\eqref{eq:cC+}, we can write
\begin{equation*}
\begin{split}
	& \sumtwo{z \in \Delta\N_0}{|z - y_n| > M}
	\frac{F^+_{n-{\bar k}}(x_n, [z,z+1))}
	{\frac{1}{a_n} \bbP(\tau_1^- > n) V^-(x_n)} \, \overline{f}_\Delta^+(z,y_n)
	\,\le\, \hat c
	\sumtwo{z \in {\Delta}\Z}{|z - y_n| > {M}}
	\Delta \, \sup_{u \in [y_n - z, y_n - z+\Delta)} f_{\bar k}(u) \\
	& \qquad \qquad \,=\, \hat c \sumtwo{w \in {\Delta}\Z + y_n}{|w| > { M}}
	\Delta \, \sup_{u \in [w, w + \Delta)} f_{\bar k}(u)
	\,\le\, 2 \hat c \sumtwo{w \in {\Delta}\Z}{|w| > { M}}
	\Delta \, \sup_{u \in [w, w + \Delta)} f_{\bar k}(u)  \,,
\end{split}
\end{equation*}
where the factor 2 in the last inequality is due to the lattice shift, from
${\Delta}\Z + y_n$ to ${\Delta}\Z$.
Recalling \eqref{eq:Corig} and \eqref{eq:dRihyp}, it follows that
the last sum is convergent, hence we can choose $M$ large enough
so that it is less than $\epsilon/2$.
Let us now focus on the contribution to \eqref{eq:daidaidai}
of the terms with $|z - y_n| \le M$. Note that there are
only a finite number of such terms.
Since each $z$ with $|z - y_n| \le M$ is such that $z/a_n \to +\infty$, it follows
by \eqref{eq:asi1} that
\begin{equation*}
	\lim_{n\to\infty} \bigg( \sup_{z\in \Delta\N_0,\, |z - y_n| \le M}
	\frac{F^+_{n-{\bar k}}(x_n, [z,z+1))}
	{\frac{1}{a_n} \bbP(\tau_1^- > n) V^-(x_n)} \bigg) = 0 \,.
\end{equation*}
By construction the function
$f_{\bar k}$ is bounded, hence there exists $n_4$ such that
for $n \ge n_4$
\begin{equation*}
	\sumtwo{z \in \Delta\N_0}{|z - y_n| \le M}
	\frac{F^+_{n-{\bar k}}(x_n, [z,z+1))}
	{\frac{1}{a_n} \bbP(\tau_1^- > n) V^-(x_n)} \, \overline{f}_\Delta^+(z,y_n)
	\,\le\, \frac{\epsilon}{2} \,.
\end{equation*}
Recalling \eqref{eq:daidaidai}, it follows that \eqref{eq:atlaast}
holds true for $n\ge n_4$, completing the proof.

\section{Proof of Theorem \ref{th:main}}
\label{sec:invpr}

This section is devoted to the proof of the invariance principle
in Theorem \ref{th:main}.
We recall that, by Hypothesis~\ref{hyp:main},
$(S = \{S_n\}_{n\ge 0}, \bbP)$ is a random walk in the domain of attraction
of a L\'evy process $(X = \{X_t\}_{t\ge 0}, \bP)$ with index $\ga \in (0,2]$ and positivity
parameter $\rho \in (0,1)$. We denote by $(a_n)_{n\in\N} \in R_{1/\ga}$ the norming
sequence, so that $S_n/a_n \Rightarrow X_1$.




\subsection{Random walks conditioned to stay positive}\label{srwn}

We remind that for convenience we assume that $\bbP$ is a law on the space
$\Omega^{RW} := \R^{\N_0}$, $S = \{S_n\}_{n\in\N}$ is the coordinate
process on this space and $\bbP_x$ the law on $\Omega_{RW}$ of
$S+x$ under $\bbP$, for all $x \in \R$ (but we only consider the case $x\in\Z$).
We also set $\Omega^{RW}_N := \R^{\{0, \ldots, N\}}$ for $N\in\N_0$.

Let us recall the definitions of
the bridges of length $N$ of the random walk $\bbP$
from $x$ to $y$ conditioned to stay non-negative or strictly positive,
cf. \eqref{eq:bbPN+0}, \eqref{eq:hatbbPN+0}:
these are the laws $\bbP_{x,y}^{\uparrow,N}$ and
$\widehat\bbP_{x,y}^{\uparrow,N}$ on $\Omega_N^{RW}$,
defined for $x,y \in \N_0$ and $N \in \N$ by
\begin{align} \label{eq:bbPN+}
	\bbP_{x,y}^{\uparrow,N}(\,\cdot\,) & :=
	\bbP_x( \,\cdot\, | S_1 \ge 0, \ldots, S_{N-1} \ge 0, S_N = y) \,, \\
	\label{eq:hatbbPN+}
	\widehat\bbP_{x,y}^{\uparrow,N}(\,\cdot\,) & :=
	\bbP_x( \,\cdot\, | S_1 > 0, \ldots, S_{N-1} > 0, S_N = y) \,.
\end{align}
Other basic laws on $\Omega_{RW}$ are $\bbP_x^\uparrow$ and $\widehat\bbP_x^\uparrow$,
the laws of the random walk $\bbP$ started at $x$
and conditioned to stay non-negative
or strictly positive for all time (cf. \cite{cf:BerDon,cf:CarCha}):
these are defined for $x \in \N_0$
by setting, for all $N\in\N$ and $B \in \gs(S_0, \ldots, S_N)$,
\begin{align} \label{eq:bbP+}
	\bbP_x^\uparrow(B) & := \frac{1}{V^-(x)} \bbE_x(\ind_B \, V^-(S_N) \,
	\ind_{\{S_1\ge 0, \ldots, S_N \ge 0\}}) \,, \\
	\label{eq:hatbbP+}
	\widehat\bbP_x^\uparrow(B) & := \frac{1}{\underline {V}^-(x)}
	\bbE_x(\ind_B \, \underline {V}^-(S_N) \,
	\ind_{\{S_1>0, \ldots, S_N>0\}}) \,,
\end{align}
where the renewal functions $V^-(\cdot)$ and $\underline {V}^-(\cdot)$
have been introduced in \S\ref{sec:fluctutation}. Lemma~\ref{th:harm} and the following lines
guarantee that the laws $\bbP_x^\uparrow$ and $\widehat\bbP_x^\uparrow$ are well-defined.
We have already observed that
$\bbP_{x,y}^{\uparrow,N}$ and $\widehat\bbP_{x,y}^{\uparrow,N}$ may be viewed
as bridges of the laws $\bbP^\uparrow_x$ and $\widehat\bbP^\uparrow_x$,
respectively: more precisely, we may write
\begin{equation} \label{eq:abscont}
	\bbP_{x,y}^{\uparrow,N}(\,\cdot\,) = \bbP_x^\uparrow(\,\cdot\,|S_N=y) \,,
	\qquad
	\widehat\bbP_{x,y}^{\uparrow,N}(\,\cdot\,) =
	\widehat\bbP_x^\uparrow(\,\cdot\,|S_N=y) \,.
\end{equation}

The following lemma is a direct consequence of the time reversal property of random walks. Let
$\tilde{\bbP}_{x,y}^{\uparrow,N}$ be the law of the bridge of
the reflected walk $\widetilde S := -S$ conditioned to stay non-negative, as it is defined
in (\ref{eq:bbPN+}) for $S$.
\begin{lemma}\label{eq:reversal1}
For all $x,y\ge0$, and for all $N\ge1$, under the law $\bbP_{x,y}^{\uparrow,N}$, the process
$(S_N-S_{N-M},\,0\le M\le N)$
has law $\tilde{\bbP}_{y,x}^{\uparrow,N}$.
\end{lemma}


\subsection{L\'evy processes conditioned to stay positive}\label{slpn}
\label{sec:levypos}

We recall that $\Omega := D([0,\infty),\R)$ is the space of real-valued
c\`adl\`ag paths which are defined on $[0,\infty)$, $X = \{X_t\}_{t\ge 0}$ is
the corresponding coordinate process and $\bP$ is the law
on $\Omega$ under which $X$ is the stable L\'evy process
appearing in Hypothesis~\ref{hyp:main}. We denote by $\bP_a$
the law on $\Omega$ of $X+a$ under $\bP$, for all $a\in\R$.
We also denote by $\Omega_t := D([0,t],\R)$ for $t \ge 0$ the space of paths of length~$t$.

In analogy to the discrete case, we can define the law of the L\'evy
process started at $a > 0$ and conditioned to stay positive for all time
\cite{cf:Cha97,cf:CarCha} to
be the law $\bP_a^\uparrow$ on $\Omega$ such that, for all $t > 0$
and $B \in \gs(X_{s}, 0 \le s \le t)$,
\begin{equation} \label{eq:bP+}
	\bP_a^\uparrow(B) := \frac{1}{U^-(a)} \bE_a
	(\ind_B \, U^-(X_t) \, \ind_{\{X_s \ge 0, \, \forall 0 \le s \le t\}} ) \,,
\end{equation}
where $U^-(\cdot)$ is the renewal function associated to the descending
ladder height process of $(X,\bP)$. Since $X$ is stable
we have $U^-(x) = x^{\ga(1-\rho)}$.
Note that for L\'evy processes
there is no distinction between staying non-negative
and strictly positive.
Although \eqref{eq:bP+} does not make sense when $a=0$,
because $U^-(0)=0$, the law $\bP_0^\uparrow$ can still be defined, see
\cite{cf:Cha97}, and we have
\begin{equation} \label{eq:bP0}
\bP_a^\uparrow \Rightarrow\bP_0^\uparrow\,,\;\;\;\mbox{as $a \downarrow 0$,}
\end{equation}
where $\Rightarrow$ denotes weak convergence on $\Omega$.

Then we will define the law $\bP^{\uparrow,t}_{a,b}$,
on $\Omega_t$ of the bridge of the L\'evy process $(X,\bP)$, with length $t>0$, between $a\ge0$ and $b\ge0$,
conditioned to stay positive. The following definitions and results are mainly excerpt from \cite{cf:Ur11}
to which we refer for details.
Set $\tilde{X}:=-X$ and denote by $\tilde{U}^-(x)=x^{\alpha\rho}$ the renewal function associated to the
descending ladder heights process of $(\tilde{X},\bP)$. Define the measure
\begin{equation}\label{eq:lambdaup}
	\lambda^\uparrow({\dd z})={U}^-(z)\tilde{U}^-(z)\,\dd z=z^\alpha\,\dd z\,,
\end{equation}
on $[0,\infty)$ and let $g_t^\uparrow(a,b)$ be the unique version of the semigroup density of $\bP^\uparrow$
with respect to the measure $\lambda^\uparrow({\dd z})$, i.e.
\[g_t^\uparrow(a,b)\,\lambda^\uparrow({\dd b}) := \bP_a^\uparrow(X_t \in \dd b)\,,\]
which satisfies the Chapman-Kolmogorov equation:
\begin{equation}\label{eq:chapman-kolmogorov}
g_{s+t}^\uparrow(a,b)=\int_0^\infty g_s^\uparrow(a,z)g_t^\uparrow(z,b)\,\lambda^\uparrow(\dd z)\,,\;\;\;
\mbox{for all $s,t>0$ and $a\ge0$, $b\ge0$.}
\end{equation}
By (\ref{eq:bP+}), for $a,b>0$, this density may be written as
\begin{equation}\label{eq:g+}
	g_t^\uparrow(a,b) := \frac{1}{U^-(a)\tilde{U}^-(b)}g_t^+(a,b) \,,
\end{equation}
where $g_t^+(a,b)$ is the semigroup of the L\'evy process $(X,\bP)$ killed at its first passage time
below 0, i.e. for $a,b>0$,
\begin{equation*}
	g_t^+(a,b)\,{\dd b} := \bP_a(X_t \in \dd b,\, X_s \ge 0, \, \forall 0 \le s \le t)  \,.
\end{equation*}
By Lemma 3 in \cite{cf:Ur11}, for each $t>0$, the densities $g_t^\uparrow(a,b)$ are strictly positive and continuous on $[0,\infty)\times[0,\infty)$ (including
$a=0$ and $b=0$).
Then for all $a\ge0$ and $b\ge0$, the bridge of the L\'evy process conditioned to stay positive is formally defined as
follows: for $\gep > 0$ and $B \in \gs(X_s, 0 \le s \le t-\gep)$,
\begin{equation}\label{eq:bPab+}
	\bP^{\uparrow,t}_{a,b}(B) := \frac{1}{g_t^\uparrow(a,b)}
	\bE_a^\uparrow(\ind_B \, g_{\gep}^\uparrow(X_{t-\gep},b)) \,,
\end{equation}
and we may check, thanks to the Chapman-Kolmogorov equation (\ref{eq:chapman-kolmogorov}), that this relation indeed defines a
regular version of the conditional law $\bP^{\uparrow}_{a}(\,\cdot\,|\,X_t=b)$ on $\Omega_t$. Moreover, from \cite{cf:Ur11},
the measures $\bP^{\uparrow,t}_{a,b}$ are weakly continuous in $a$ and $b$ on $\Omega_t$.

Let us denote by $\tilde{\bP}^{\uparrow,t}_{a,b}$ the law of the bridge of $(\tilde{X},\bP)=(-X,\bP)$ conditioned to stay positive. Then
we derive, from the duality property which is proved in Lemma 1
of \cite{cf:Ur11} and from Corollary 1 in \cite{cf:FitPitYor}, the following
time reversal property, which is the continuous time counterpart of Lemma \ref{eq:reversal1}.
\begin{lemma}\label{eq:reversal2}
With the convention $0-=0$, for all $a,b\ge0$, and for all $t>0$, under the law ${\bP}^{\uparrow,t}_{a,b}$, the process
$(X_t-X_{(t-s)-},\,0\le s\le t)$
has law $\tilde{\bP}^{\uparrow,t}_{b,a}$.
\end{lemma}
\noindent Now let us focus on the special case where $a=b=0$. In \cite{cf:Cha97}, Lemma 2, the law $\bP^{\uparrow,t}_{0,0}$ is
interpreted as the weak limit
\[\bP^{\uparrow,t}_{0,0}(\cdot)=\lim_{\varepsilon\downarrow0}\bP^{\uparrow}_{0}(\,\cdot\,|\,0\le X_t\le\varepsilon)\,,\]
and, when the L\'evy process $(X,\bP)$ has no negative jumps,
this law is identified to the law of the normalized excursion
of the reflected process at its past infimum. In particular,
when $X$ is the standard
Brownian motion, it corresponds to the normalized Brownian excursion.

Then for $0 < \epsilon < t$ and $x \ge 0$ we set
\begin{equation} \label{eq:fgept}
	f_{\gep,t}(x) := \frac{g_{\gep}^\uparrow(x,0)}{g_t^\uparrow(0,0)} \,,
\end{equation}
so that, by (\ref{eq:bPab+}), $f_{\gep,t}(X_{t-\epsilon})$ is
the Radon-Nikodym density of
$\bP^{\uparrow,t}_{0,0}$ with respect to $\bP^{\uparrow}_{0}$ on
the sigma field $\gs(X_s, 0 \le s \le t-\gep)$, i.e. for $B \in \gs(X_s, 0 \le s \le t-\gep)$,
\begin{equation}\label{eq:bP00+}
	\bP^{\uparrow,t}_{0,0}(B) :=
	\bE_0^\uparrow(\ind_B \, f_{\varepsilon,t}(X_{t-\gep},0)) \,.
\end{equation}
We recall that $g^-(\cdot)$ is the density of the law of the terminal value of
the meander with length $1$ of $-X$.
Recall also from (\ref{eq:DonSav}) and (\ref{eq:C+-}) that
$g^-(x)\sim{\mathtt C}^-g(0)x^{\alpha(1-\rho)}$, as $x\downarrow0$.
Then we have the following result, proved in Appendix~\ref{ap:B}.

\begin{lemma}\label{th:density} The function $f_{\gep,t}(\cdot)$ is continuous on $[0,\infty)$
and is given by:
\begin{align}
	f_{\gep,t}(x) & \,=\, \frac{(t/\epsilon)^{1+1/\alpha}}{\mathtt{C}^-\, g(0)}
	\, \frac{g^-(\epsilon^{-1/\alpha}x)}{(\epsilon^{-1/\alpha}x)^{\alpha(1-\rho)}}
	\,, \quad \text{for } x> 0 \,,\label{eq:density}\\
	f_{\gep, t}(0) & \,=\, \left(\frac{t}{\gep}\right)^{1+\frac1\alpha}\,.\label{eq:asfgep}
\end{align}
\end{lemma}
\noindent In the sequel we will simply denote $f_{\gep}(x):=f_{\gep,1}(x)$.

\begin{remark}\rm \label{rem:Gauss2}
In the Brownian case $\ga = 2$, $\rho = \frac 12$
everything is explicit.
The density of the Brownian meander is
$g^-(x) = x\, e^{-x^2/2}\, \ind_{(0,\infty)}(x)$,
while the density of $\bP_{0}^\uparrow$ (the Bessel(3) process) at time $t$ is
\begin{equation*}
	\frac{\bP_{0}^\uparrow(X_t \in \dd x)}{\dd x} =
	\sqrt{\frac{2}{\pi t}} \, \frac{x^2}{t} \, e^{-x^2/(2t)}
	= \sqrt{\frac{2}{\pi}} \, \frac{x}{t} \, g^- \bigg( \frac{x}{\sqrt t} \bigg) \,,
\end{equation*}
cf. \cite[\S3 in Chapter~VI]{cf:RevYor}.
Since $\lambda^\uparrow(\dd z) = z^2 \, \dd z$
(recall \eqref{eq:lambdaup}), we obtain
\begin{equation*}
	g_t^\uparrow(0,x) \,:=\, \frac{\bP_{0}^\uparrow(X_t \in \dd x)}
	{\lambda^\uparrow(\dd x)} \,=\,
	\sqrt{\frac{2}{\pi}} \, \frac{1}{t^{3/2}} \, e^{-x^2/(2t)} \,.
\end{equation*}
By symmetry $g_t^\uparrow(x,0) = g_t^\uparrow(0,x)$,
and recalling \eqref{eq:fgept} we find
\begin{equation*}
	f_{\epsilon, t}(x) \,=\, \frac{g_\epsilon^\uparrow(x,0)}{g_t^\uparrow(0,0)}
	\,=\, \frac{t^{3/2}}{\epsilon^{3/2}} \, e^{-x^2/(2\epsilon)} \,.
\end{equation*}
which coincides precisely with the expression in \eqref{eq:density},
because $g(0)=1/\sqrt{2\pi}$ and ${\mathtt C}^\pm = \sqrt{2\pi}$, as it was
shown in Remark~\ref{rem:onlygauss?}.

Also note that the density of $\bP_{0,0}^{\uparrow,1}$ (the Brownian excursion of length~$1$)
at time $1-\gep$ is
\begin{equation*}
\begin{split}
	\frac{\bP_{0,0}^{\uparrow,1}(X_{1-\gep}\in\dd x)}{\dd x}
	& = \frac{2}{\sqrt{2\pi}} \, \frac{1}{1-\gep}
	g^-\bigg(\frac{x}{\sqrt{1-\gep}}\bigg) \,
	\frac{1}{\gep}
	g^-\bigg(\frac{x}{\sqrt{\gep}}\bigg) \,,
\end{split}
\end{equation*}
cf. \cite[\S3 in Chapter~IX]{cf:RevYor}.
Recalling that $f_{\gep}(x) = f_{\gep,1}(x)$ is the Radon-Nikodym derivative
of the law $\bP_{0,0}^{\uparrow,1}$
with respect to $\bP_{0}^\uparrow$ at time $1-\epsilon$,
we can write
\begin{equation*}
	f_{\gep}(x) \,=\,
	\frac{\bP_{0,0}^{\uparrow,1}(X_{1-\gep}\in\dd x)}{\bP_{0}^\uparrow(X_{1-\gep}\in\dd x)}
	= \frac{1}{\gep \, x} \, g^{-} \bigg( \frac{x}{\sqrt{\gep}} \bigg)
	= \frac{1}{\epsilon^{3/2}} \, e^{-x^2/(2\epsilon)} \,,
\end{equation*}
finding again the expression
\eqref{eq:density} (in the special case $t=1$).
\end{remark}


\subsection{Proof of Theorem~\ref{th:main}}
For simplicity, we restrict our attention to the lattice case,
i.e. we assume that the law of $S_1$ is supported by $\Z$ and is aperiodic
(recall Hypothesis \ref{hyp:main2}). The proof in the absolutely continuous case
is almost identical, except that the local limit theorems of Proposition~\ref{th:xq}
must be replaced by the corresponding ones of Proposition~\ref{th:xqabscont}.

Let us fix $a, b \in [0,\infty)$ and $(x_N)_{N\in\N}, (y_N)_{N\in\N}$
sequences in $\N_0$ such that $x_N/a_N \to a$ and $y_N / a_N \to b$
as $N\to\infty$. We
recall that $\phi_N: \Omega^{RW} \to \Omega$ is the map defined by
\begin{equation*}
	(\phi_N(S))(t) := \frac{S_{\lfloor Nt\rfloor}}{a_N} \,,
\end{equation*}
for $N\in\N$, and by extension we still denote by $\phi_N$ the analogous map
defined from $\Omega^{RW}_{N}$ to $\Omega_1$,
or more generally from $\Omega^{RW}_{NT}$ to $\Omega_T$, for any fixed $T>0$.
Our goal is to prove \eqref{eq:main},
and we only focus on the second relation, as the first one
follows exactly the same lines.

\smallskip

When both $a>0$ and $b>0$, there is nothing to prove,
as \eqref{eq:main} follows directly by Liggett's invariance
principle for the bridges~\cite{cf:Lig}. In fact, the latter states that
as $N \to\infty$
\begin{equation} \label{eq:Liggett}
	\bbP^{N}_{x_N, y_N} \circ \phi_N^{-1} \; \Longrightarrow \; \bP^1_{a,b} \,,
\end{equation}
where $\bbP^{N}_{x, y}(\cdot) := \bbP_x(S \in \cdot\, | S_N = y)$
and $\bP^1_{a,b}(\cdot) := \bP_a(X \in \cdot\,|X_1 = b)$ are the bridges
of the random walk and L\'evy process respectively.
Note that we can write
\begin{equation*}
	\widehat\bbP^{\uparrow, N}_{x_N, y_N}(\cdot) \,=\,
	\bbP^{N}_{x_N, y_N}(\cdot\, | S_1 > 0, \ldots, S_N > 0) \,, \qquad
	\bP^{\uparrow, 1}_{a,b} (\cdot) \,=\,
	\bP^{1}_{a, b} \Big( \cdot\, \Big| \inf_{0 \le t \le 1}X_t \ge 0 \Big) \,.
\end{equation*}
Since for $a,b > 0$ the conditioning to stay positive has a non-vanishing probability
under $\bP^{1}_{a, b}$, relation \eqref{eq:main} follows from \eqref{eq:Liggett}.

\smallskip

We now focus on the case $a=b=0$.
The cases $a=0$, $b>0$ and $a>0$, $b=0$ are similar and simpler, so we skip
them for brevity. By \eqref{eq:abscont}, the law $\widehat\bbP_{x,y}^{\uparrow,N}$ is absolutely
continuous with respect to $\widehat\bbP_{x}^{\uparrow}$: more precisely,
recalling relations \eqref{eq:q} and \eqref{eq:hatbbP+},
for all $0 < M < N$ and
for all $B \in \gs(S_0, \ldots, S_M)$, we can write
\begin{equation} \label{eq:basicabscont}
	\widehat\bbP_{x,y}^{\uparrow,N}(B) = \frac{1}{\widehat\bbP_x^\uparrow(S_N=y)}
	\bbE_x^\uparrow(\ind_B \, \widehat\bbP_{S_M}^\uparrow(S_{N-M}=y))
	= \frac{\underline {V}^-(x)}{\widehat q^+_N(x,y)}
	\bbE_x^\uparrow \bigg( \ind_B \,
	\frac{\widehat q^+_{N-M}(S_M, y)}{\underline {V}^-(S_M)} \bigg) \,,
\end{equation}
where we recall for clarity that
$\widehat q^+_n(x,y) := \bbP_x(S_1 > 0, \ldots, S_{n-1}>0, S_n=y)$.
If we introduce for convenience the laws on $\Omega$ and $\Omega_1$ respectively, given by
\begin{equation} \label{eq:pushlaws}
	\cP_x^{\uparrow, (N)} := \widehat\bbP_{a_N x}^{\uparrow} \circ \phi_N^{-1} \,, \qquad
	\cP_{x,y}^{\uparrow, (N)} :=
	\widehat\bbP_{a_N x,\, a_N y}^{\uparrow, N} \circ \phi_N^{-1} \,,
\end{equation}
then from \eqref{eq:basicabscont} with $M=\lfloor(1-\gep)N\rfloor$, ($N\ge2$, $0<\varepsilon<1/2$),
we infer that $\cP_{x/a_N,\, y/a_N}^{\uparrow, (N)}$ is absolutely
continuous with respect to $\cP_{x/a_N}^{\uparrow, (N)}$ (restricted to $\Omega_1$)
on the $\gs$-field $\gs(X_s, 0 \le s \le N^{-1}\lfloor(1-\gep)N\rfloor)$,
for all $\gep > 0$. More precisely, for $B \in \gs(X_s, 0 \le s \le N^{-1}\lfloor(1-\gep)N\rfloor)$
we can write
\begin{equation} \label{eq:basicabscont'}
	\cP_{x_N/a_N,\, y_N/a_N}^{\uparrow, (N)}(B) =
	\cE_{x_N/a_N}^{\uparrow, (N)} (\ind_B \,
	f^{(N)}_{\gep}(X_{N^{-1}\lfloor(1-\gep)N\rfloor})) \,,
\end{equation}
where $\cE_{x}^{\uparrow, (N)}$ denotes the expectation under $\cP_x^{\uparrow, (N)} $ and with
$\varepsilon(N)=N-\lfloor(1-\varepsilon)N\rfloor$,
\begin{equation} \label{eq:deffN}
	f^{(N)}_{\gep}(z) := \frac{\underline {V}^-(x_N)}{\widehat q^+_N(x_N,y_N)}\,
	\frac{\widehat q^+_{\varepsilon(N)}(\lfloor z a_N \rfloor , y_N)}
	{\underline {V}^-( z a_N )} \,.
\end{equation}
Recall the definition of $f_{\gep} := f_{\gep, 1}$ in (\ref{eq:density}).
We state the following basic lemma.

\begin{lemma} \label{th:uniff}
The following uniform convergence holds:
\begin{equation}\label{eq:unifconv}
	\lim_{N\rightarrow\infty} \
	\sup_{z\in\R} |f^{(N)}_{\gep}(z) - f_{\gep}(z)| = 0 \,.
\end{equation}
\end{lemma}

The proof is given below.
To complete the proof of Theorem~\ref{th:main}, we first prove that, for all $\gep>0$
and for every bounded and continuous functional $F$ on $\Omega_1$ which is
measurable with respect to $\gs(X_s, 0 \le s \le 1-\gep)$, one has
\begin{equation} \label{eq:conv}
	\lim_{N\to\infty} \cE_{x_N/a_N,y_N/a_N}^{\uparrow, (N)} (F) \,=\,
	\bE_{0,0}^{\uparrow,1}(F)  \,.
\end{equation}
Let $\varepsilon'<\varepsilon$ and $N$ sufficiently large so that
$1-\varepsilon<N^{-1}\lfloor(1-\varepsilon')N\rfloor$. Then from
(\ref{eq:bP00+}),
\begin{eqnarray*}
\cE_{x_N/a_N,y_N/a_N}^{\uparrow, (N)} (F)&=&
\cE_{x_N/a_N}^{\uparrow, (N)} (f^{(N)}_{\gep'}(X_{N^{-1}\lfloor(1-\varepsilon')N\rfloor}) \cdot F)\,.
\end{eqnarray*}
On the other hand, from (\ref{eq:basicabscont'}) and the Markov property,
\begin{eqnarray*}
\bE_{0,0}^{\uparrow,1}(F)&=&\bE_{0}^\uparrow(f_{\gep}(X_{1-\gep}) \cdot F)\\
&=&\bE_{0}^\uparrow(f_{\gep'}(X_{1-\gep'}) \cdot F)\,.
\end{eqnarray*}
Looking at (\ref{eq:inv+}), we see that (\ref{eq:conv}) is equivalent to showing that
\begin{equation} \label{eq:conv1}
	\lim_{N\to\infty} \cE_{x_N/a_N}^{\uparrow, (N)} (f^{(N)}_{\gep'}(X_{N^{-1}\lfloor(1-\varepsilon')N\rfloor}) \cdot F)
	= \bE_{0}^\uparrow(f_{\gep'}(X_{1-\gep'}) \cdot F)\,.
\end{equation}
By \eqref{eq:unifconv}, for every $\eta > 0$ there exists $N_0 < \infty$
such that $|f^{(N)}_{\gep'}(z) - f_{\gep'}(z)| \le \eta$, for all
$N\ge N_0$ and $z \in [0,\infty)$. It follows that for $N \ge N_0$,
\begin{eqnarray*}
	&&| \cE_{x_N/a_N}^{\uparrow, (N)} (f^{(N)}_{\gep'}(X_{N^{-1}\lfloor(1-\varepsilon')N\rfloor}) \cdot F)
	- \bE_{0}^\uparrow(f_{\gep'}(X_{1-\gep'}) \cdot F) |\nonumber\\
 &&\qquad\qquad \qquad\le\eta + | \cE_{x_N/a_N}^{\uparrow, (N)} (f_{\gep'}(X_{N^{-1}\lfloor(1-\varepsilon')N\rfloor}) \cdot F)-
\bE_{0}^\uparrow(f_{\gep'}(X_{1-\gep'}) \cdot F) |\,.
\end{eqnarray*}
By (\ref{eq:inv+}) and Skorokhod's representation theorem, there exist processes $Y^{(N)}$, $Y$ on a probability space
$(\Omega',\mathcal{F}, \p)$, such that $(Y^{(N)},\p)
\overset{d}{=}(X,\cP_{x_N/a_N}^{\uparrow, (N)})$, $(Y,\p)\overset{d}{=}(X,\bP_{0}^\uparrow)$, and such
that $Y^{(N)}$ converges $\p$--a.s.~toward $Y$. Since $f_{\gep'}(\cdot)$ is continuous and $Y$ is $\p$--a.s.
continuous at time $1-\gep'$, the sequence $f_{\gep'}(Y^{(N)}_{N^{-1}\lfloor(1-\varepsilon')N\rfloor})$ converges $\p$--a.s.
toward $f_{\gep'}(Y_{1-\gep'})$, so that by dominated convergence (recall that $F$ is bounded), for all $N\ge N_1$,
 \begin{eqnarray*}
&& | \e(f_{\gep'}(Y^{(N)}_{N^{-1}\lfloor(1-\varepsilon')N\rfloor}) \cdot F)-
\e(f_{\gep'}(Y_{1-\gep'}) \cdot F)|=\\
&&| \cE_{x_N/a_N}^{\uparrow, (N)} (f_{\gep'}(X_{N^{-1}\lfloor(1-\varepsilon')N\rfloor}) \cdot F)-
\bE_{0}^\uparrow(f_{\gep'}(X_{1-\gep'}) \cdot F) |\le \eta\,.
\end{eqnarray*}
Hence for $N\ge\max(N_0,N_1)$, we have
\begin{equation*}
	| \cE_{x_N/a_N}^{\uparrow, (N)} (f^{(N)}_{\gep'}(X_{N^{-1}\lfloor(1-\varepsilon')N\rfloor}) \cdot F)
	- \bE_{0}^\uparrow(f_{\gep'}(X_{1-\gep'}) \cdot F) | \le
	2 \eta \,.
\end{equation*}
Since $\eta > 0$ was arbitrary, \eqref{eq:conv1} is proved.

Relation \eqref{eq:conv}
shows that the sequence of probability distributions $(\cP_{x_N/a_N,y_N/a_N}^{\uparrow, (N)})$
restricted to $\Omega_{1-\varepsilon}$ converges weakly on this space, as $N\rightarrow+\infty$ toward $(\bP_{0,0}^{\uparrow,1})$,
for all $\varepsilon\in(0,1)$. In particular, the sequence of probability distributions $(\cP_{x_N/a_N,y_N/a_N}^{\uparrow, (N)})$
converges on $\Omega_{1}$ in the sense of finite dimensional distributions toward $(\bP_{0,0}^{\uparrow,1})$, as
$N\rightarrow+\infty$. Then in order to prove the weak convergence of this sequence, it remains to check
that it is tight on $\Omega_1$. From Theorem 15.3 of Billingsley \cite{cf:Bil}, it suffices to show that for all $\eta>0$,
\begin{equation}\label{eq:tight}
   \lim_{\delta\rightarrow0}\limsup_{N\rightarrow+\infty}\cP_{x_N/a_N,y_N/a_N}^{\uparrow, (N)}\left(\sup_{s,t\in[1-\delta,1]}|X_t-X_s|>\eta\right)=0\,.
\end{equation}
But, this follows from the time reversal properties, i.e. Lemma \ref{eq:reversal1} and Lemma \ref{eq:reversal2}. Indeed, by Lemma~\ref{eq:reversal1},
under the law $\cP_{x_N/a_N,y_N/a_N}^{\uparrow, (N)}$ the process
$(X_t-X_{(t-s)-},\,0\le s\le t)$ has law $\tilde{\cP}_{y_N/a_N,x_N/a_N}^{\uparrow, (N)}$, with an obvious notation.
Moreover,  from what have been proved above applied to $-S$ and $-X$, we obtain that
the sequence of probability distributions $(\tilde{\cP}_{y_N/a_N,x_N/a_N}^{\uparrow, (N)})$ restricted to $\Omega_{1-\varepsilon}$
converges weakly on this space, as $N\rightarrow+\infty$ toward $(\tilde{\bP}_{0,0}^{\uparrow,1})$, for all $\varepsilon\in(0,1)$.
Hence from Theorem 15.3 of Billingsley \cite{cf:Bil}, we have
\begin{equation}
   \lim_{\delta\rightarrow0}\limsup_{N\rightarrow+\infty}\tilde{\cP}_{y_N/a_N,x_N/a_N}^{\uparrow, (N)}\left(\sup_{s,t\in[0,\delta]}|X_t-X_s|>\eta\right)=0\,,
\end{equation}
which is precisely (\ref{eq:tight}), thanks to our time reversal argument.

\begin{proof}[Proof of Lemma~\ref{th:uniff}]

From the argument recalled in section \ref{sec:notation},
relation \eqref{eq:unifconv} is equivalent to the convergence
\begin{equation}\label{eq:unifconvsimple}
	\lim_{N\rightarrow\infty} \ 	|f^{(N)}_{\gep}(z_N) - f_{\gep}(z_N)| = 0 \,,
\end{equation}
\emph{for every $\overline{z} \in [0,+\infty]$ and for every
sequence $(z_N)_{N\in\N}$ in $[0,\infty)$ such that $z_N \to \overline{z}$}.

We start with the case $\overline{z} \in (0,\infty)$.
Since by assumption both $x_N, y_N = o(a_N)$, by
the second relation in \eqref{eq:xqsmall} we have as $N\to\infty$.
\begin{equation} \label{eq:barr}
	\widehat q^+_N(x_N,y_N) \sim \underline {V}^-(x_N) \,
	\underline {\widehat V}^+(y_N)
	\, \frac{g(0)}{N \, a_N} \,.
\end{equation}
By the second relation in \eqref{eq:xqlargeinv} we get
\begin{equation} \label{eq:barr2}
	\widehat q^+_{\varepsilon(N)}(\lfloor z_N a_N \rfloor , y_N) \sim
	\underline {\widehat V}^+(y_N) \, \frac{\bbP(\widehat \tau_1^+ > \varepsilon(N))}{a_{\varepsilon(N)}} \,
	g^-\bigg( z_N \frac{a_N}{a_{\varepsilon(N)}} \bigg) \,,
\end{equation}
where we have applied relations \eqref{eq:underVhatV} and \eqref{eq:usefuly}.
Therefore we can write, as $N\to\infty$
\begin{equation} \label{eq:auxf}
	f^{(N)}_{\gep}(z_N) \sim \frac{1}{g(0)} \, \frac{a_N}{a_{\varepsilon(N)}} \,
	\frac{N \, \bbP(\widehat \tau_1^+ >\varepsilon(N))}
	{\underline {V}^-(z_N a_N)} \,
	g^-\bigg( z_N \frac{a_N}{a_{\varepsilon(N)}} \bigg)\,.
\end{equation}
We know from \S\ref{sec:conseq}, cf. in particular \eqref{eq:sumup1} and \eqref{eq:sumup2}, that
\begin{equation*}
	a_N \in R_{1/\ga} \,, \qquad
	\underline {V}^-(\cdot) \in R_{\ga(1-\rho)} \,, \qquad
	\bbP(\widehat \tau_1^+ > N) \in R_{-\rho} \,,
\end{equation*}
therefore as $N\to\infty$, since $z_N \to \overline z \in (0,\infty)$,
\begin{equation*}
	a_{\varepsilon(N)} \sim \gep^{1/\ga} a_N \,, \quad
	\underline {V}^-(z_N a_N) \sim {\overline z}^{\ga(1-\rho)} \,
	\underline {V}^-(a_N) \,, \quad
	\bbP(\widehat \tau_1^+ > \varepsilon(N)) \sim \gep^{-(1-\rho)}
	\bbP(\widehat \tau_1^+ > N) \,.
\end{equation*}
Recalling \eqref{eq:lien}, \eqref{eq:underVhatV} and \eqref{eq:usefuly}, we obtain
\begin{equation} \label{eq:lala}
	N \, \bbP(\widehat \tau_1^+ > N) \sim \frac{\underline V^-(a_N)}{{\mathtt C}^-} \,.
\end{equation}
Since $z_N \to \overline{z} \in (0,\infty)$ by assumption,
from the preceding relation we get
\begin{equation} \label{eq:limZ}
	\lim_{N\to\infty} f^{(N)}_{\gep}(z_N) =
	\frac{1}{{\mathtt C}^- \, g(0)}
	\, \gep^{-1} \, \gep^{-1/\ga}\, \frac{g^-(\gep^{-1/\ga}\, \overline{z})}
	{(\gep^{-1/\ga}\overline{z})^{\ga(1-\rho)}}
	= f_\epsilon(\overline{z}) = \lim_{N\to\infty} f_\epsilon(z_N)\,,
\end{equation}
where the second equality follows by \eqref{eq:C+-} and the definition \eqref{eq:density}
of $f_\epsilon = f_{\epsilon,1}$.
We have shown that \eqref{eq:unifconvsimple}
holds true when $\overline{z} \in (0,\infty)$.

Next we consider the case $\overline{z} = 0$, so that
$z_N a_N = o(a_N)$. By the second relation in \eqref{eq:xqsmall},
we have
\begin{equation} \label{eq:barr3}
	\widehat q^+_{\varepsilon(N)}(\lfloor z_N a_N\rfloor , y_N) \sim
	\underline {V}^-(z_N a_N) \,
	\underline {\widehat V}^+(y_N)
	\, \frac{g(0)}{\varepsilon(N) \, a_{\varepsilon(N)}} \,,
\end{equation}
therefore by \eqref{eq:barr} and \eqref{eq:deffN} we obtain as $N\to\infty$
\begin{equation*}
	f^{(N)}_{\gep}(z_N) \sim \frac{a_N}{\gep a_{\varepsilon(N)}} \sim
	\frac{1}{\epsilon^{1+1/\alpha}} = f_\gep(0)
	= \lim_{N\to\infty} f_\gep(z_N) \,,
\end{equation*}
by \eqref{eq:asfgep} and the continuity of $f_\gep(\cdot)$.
This shows that \eqref{eq:unifconvsimple}
holds true also when $\overline{z} = 0$.

Finally, we have to consider the case $\overline{z} = +\infty$.
Since $g^-(z) \to 0$ as $z \to +\infty$, by the second relation in \eqref{eq:xqlargeinv}
we can write
\begin{equation} \label{eq:barr4}
	\widehat q^+_{\varepsilon(N)}(\lfloor z_N a_N \rfloor , y_N) =
	\underline {\widehat V}^+(y_N) \,
	\frac{\bbP(\widehat \tau_1^+ > \varepsilon(N))}{a_{\varepsilon(N)}} \, o(1) \,.
\end{equation}
Recalling \eqref{eq:barr} and \eqref{eq:deffN}, we obtain as $N\to\infty$
\begin{equation*}
	f^{(N)}_{\gep}(z_N) = \frac{1}{g(0)} \, \frac{a_N}{a_{\varepsilon(N)}}
	\frac{N \, \bbP(\widehat \tau_1^+ > \varepsilon(N))}{\underline {V}^-(z_N a_N)} \, o(1) \,.
\end{equation*}
Observe that $a_N/a_{\varepsilon(N)}$ is bounded and
$\underline {V}^-(z_N a_N)/\underline {V}^-(a_N) \to +\infty$,
because $z_N \to +\infty$. Recalling \eqref{eq:lala}, it follows
that $f^{(N)}_{\gep}(z_N) \to 0$, i.e. \eqref{eq:unifconvsimple}
holds also when $\overline z = +\infty$.
\end{proof}



\appendix
\section{Proof of Lemma~\ref{th:harm}}
\label{sec:proofharm}

The proof is a slight generalization of Proposition~B.1 in~\cite{cf:CarCha}.
Plainly, it is sufficient to show that \eqref{eq:V} and
\eqref{eq:V-} hold for $N = 1$, that is
\begin{gather} \label{eq:V1}
    V^-(x) = \bbE_x \big( V^-(S_1)\, \ind_{(S_1 \ge 0)} \big) \quad
    \forall x \ge 0 \,, \\
    \label{eq:V1-}
    \underline V^-(x) = \bbE_x \big( \underline V^-(S_1)\, \ind_{(S_1 > 0)} \big) \quad
    \forall x \ge 0 \,,
\end{gather}
and the general case follows by the Markov property. Also note that
from \eqref{eq:V1} we can write, for $x > 0$ and $\gep \in (0,x)$,
\begin{equation*}
    V^-(x-\gep) = \bbE \big( V^-(S_1+x-\gep)\, \ind_{(S_1 + x \ge \gep)} \big) \,.
\end{equation*}
Letting $\gep \downarrow 0$,
recalling that $\underline V^-(x) := \lim_{\gep \downarrow 0}V^-(x-\gep)$
and using monotone convergence, we obtain
for all $x>0$
\begin{equation*}
	\underline V^-(x) = \bbE \big( \underline V^-(S_1+x)\, \ind_{(S_1 + x > 0)} \big)
	= \bbE_x \big( \underline V^-(S_1)\, \ind_{(S_1 > 0)} \big) \,,
\end{equation*}
which is nothing but \eqref{eq:V1-} for $x>0$. It therefore
suffices to prove \eqref{eq:V1} and the special case of \eqref{eq:V1-} for $x=0$.
Recalling that $V^-(\cdot)$ and $\widehat V^-(\cdot)$ differ by a constant
multiple, cf. \eqref{eq:VhatV}, we are left with proving that
\begin{align} \label{eq:V1bis}
    \widehat V^-(x) & = \bbE_x \big( \widehat V^-(S_1)\, \ind_{(S_1 \ge 0)} \big), \qquad
    \forall x \ge 0 \,,\\
	\label{eq:V1-0} 	1 & = \bbE \big( \underline V^-(S_1)\, \ind_{(S_1 > 0)} \big) \,.
\end{align}

We first prove the special case $x=0$ of \eqref{eq:V1bis}:
by the definition \eqref{eq:defVhat} of $\widehat V^-(\cdot)$
and the Duality Lemma (cf. Remark~\ref{th:otherproof}), we can write
\begin{align*}
    & \bbE(\widehat V^-(S_1) \ind_{(S_1 \ge 0)})
    = \int_{(y\ge 0)} \bbP \big( S_1 \in \dd y \big) \, \widehat V^- (y) \\
    & \quad \qquad= \int_{(y\ge 0)} \bbP \big( S_1 \in \dd y \big)
    \sum_{n\ge 0} \sum_{k\ge 0} \bbP \big( \widehat \tau_k^- = n,\, S_n \ge -y \big) \\
    & \quad \qquad= \ \bbP (S_1 \ge 0) \ + \
    \sum_{n\ge 1} \int_{(y\ge 0)}
    \bbP \big( S_1 \in \dd y \big)
    \bbP \big( S_1 < 0, \ldots, S_n < 0,\, S_n \ge -y \big) \,.
\end{align*}
Summing on the values of $S_n$ and using the Markov property, we obtain
\begin{align*}
    & \bbE(\widehat V^-(S_1) \ind_{(S_1 \ge 0)}) \\
    & \quad \qquad= \ \bbP \big(\tau_1^+ = 1\big) \
    + \ \sum_{n\ge 1} \int_{(z < 0)} \bbP \big( S_1 < 0, \ldots,
    S_{n-1} < 0, S_n \in \dd z \big)
    \, \bbP \big( S_1 \ge - z \big)\\
    & \quad \qquad= \ \bbP \big( \tau_1^+ = 1\big) \ + \
    \sum_{n\ge 1} \bbP \big( S_1 < 0, \ldots, S_n < 0,\, S_{n+1} \ge 0 \big)
    = \sum_{m\in\N} \bbP \big( \tau_1^+ = m \big)\,.
\end{align*}
But $\sum_{m\in\N} \bbP ( \tau_1^+ = m ) =
\bbP \big( \tau_1^+ < \infty \big) = 1$, because by hypothesis
$\limsup_k S_k = + \infty$, $\bbP$--a.s.. Since
$\widehat V^-(0)=1$, equation \eqref{eq:V1bis} is proved for $x=0$.
Analogously,
\begin{align*}
    & \bbE( \underline V^-(S_1) \ind_{(S_1 > 0)})
    = \int_{(y> 0)} \bbP \big( S_1 \in \dd y \big) \, \underline V^- (y) \\
    & \quad \qquad= \int_{(y> 0)} \bbP \big( S_1 \in \dd y \big)
    \sum_{n\ge 0} \sum_{k\ge 0} \bbP \big( \tau_k^- = n,\, S_n > -y \big) \\
    & \quad \qquad= \ \bbP (S_1 > 0) \ + \
    \sum_{n\ge 1} \int_{(y> 0)}
    \bbP \big( S_1 \in \dd y \big)
    \bbP \big( S_1 \le 0, \ldots, S_n \le 0,\, S_n > -y \big) \\
    & \quad \qquad= \ \bbP \big(\widehat \tau_1^+ = 1\big) \
    + \ \sum_{n\ge 1} \int_{(z \le 0)} \bbP \big( S_1 \le 0, \ldots,
    S_{n-1} \le 0, S_n \in \dd z \big)
    \, \bbP \big( S_1 > - z \big)\\
    & \quad \qquad= \ \bbP \big( \widehat \tau_1^+ = 1\big) \ + \
    \sum_{n\ge 1} \bbP \big( S_1 \le 0, \ldots, S_n \le 0,\, S_{n+1} > 0 \big)
    = \sum_{m\in\N} \bbP \big( \widehat \tau_1^+ = m \big)\,.
\end{align*}
Plainly $\sum_{m\in\N} \bbP ( \widehat \tau_1^+ = m ) =
\bbP \big(\widehat \tau_1^+ < \infty \big) = \bbP \big(\tau_1^+ < \infty \big) = 1$,
and since by definition
$\underline {V}^-(0)=1$, equation \eqref{eq:V1-0} is proved.

It only remains to prove \eqref{eq:V1bis} for $x>0$.
Recalling that $\widehat V^-(x) = \bbE_x[\#\{k\ge 0: \widehat H_k^- \in [0, x]\}]$,
we can condition the random variable $\#\{k\ge 0: \widehat H_k^- \in [0, x]\}$
on $S_1$ and use the Markov property of $S$, getting for $x> 0$
\begin{align*}
    \widehat V^-(x) & \;=\; \int_\R \bbP \big( S_1 \in \dd y \big)
    \Big\{ \big(1+\widehat V^-(x+y) - \widehat V^-(y)\big)\, \ind_{(y\ge 0)}\\
    & \qquad \qquad \qquad \qquad + \ \big( 1 + \widehat V^-(x+y) \big)\,
    \ind_{(y \in [-x,0))}\ + \ \ind_{(y < -x)} \Big\}\\
    & \;=\; 1 \,+\,
    \bbE \big( \widehat V^-(S_1 + x) \, \ind_{(S_1 +x \ge 0)} \big) \, - \,
    \bbE \big( \widehat V^- (S_1) \, \ind_{(S_1 \ge 0)} \big) \\
    & \;=\; \bbE_x \big( \widehat V^-(S_1) \, \ind_{(S_1 \ge 0)} \big) \,,
\end{align*}
having used \eqref{eq:V1bis} for $x=0$. This proves \eqref{eq:V1bis} for every
$x>0$ and we are done.\qed

\section{Complements on the local limit theorems in the lattice case}
\label{ap:llt-compl}

\subsection{Proof of the second relation in \eqref{eq:xqsmall} in the case $x=y=0$}
\label{sec:45-2}

We start from a basic inequality by Alili and Doney,
cf. equation~(3) in~\cite{cf:AliDon}:
\begin{equation} \label{eq:abau}
	\widehat q_n^+(0,0) = \bbP(\tau_1^- = n, H_1^- = 0)
	= \frac 1n\, \bbP(H_1^- > 0, S_n=0) \,,
\end{equation}
which reads in general as
$\bbP(\tau_k^- = n, H_k^- = x) = \frac{k}{n} \, \bbP( H_{k-1}^- \le x < H_k^-,\, S_n = x)$
(the interchange of $<$ and $\le$ with respect to \cite{cf:AliDon} is that we
consider weak rather than strict ladder variables).
Let us show that the events $\{H_1^- > 0\}$ and $\{S_n=0\}$,
or equivalently $\{H_1^- = 0\}$ and $\{S_n=0\}$, become
asymptotically independent for large $n$. Using the Markov property, we have
\begin{equation*}
	\bbP( H_1^- = 0 | S_n = 0) = \sum_{k=1}^n \bbP( \tau_1^- = k,
	H_1^- = 0) \, \frac{\bbP(S_{n-k}=0)}{\bbP(S_n = 0)} \, \,.
\end{equation*}
We split the sum in the three ranges $1 \le k \le \sqrt{n}$,
$\sqrt{n} < k \le \frac{n}{2}$ and $\frac{n}{2} < k \le n$.
Gnedenko's local limit theorem \eqref{eq:Gnedenko} yields
$\bbP(S_m = 0) \sim g(0)/a_m$
as $m\to\infty$, therefore $\bbP(S_{n-k}=0)/\bbP(S_n = 0) \to 1$
as $n\to\infty$, uniformly in $1 \le k \le \sqrt{n}$. It follows that
as $n\to\infty$
\begin{equation*}
	\sum_{k=1}^{\sqrt{n}} \bbP( \tau_1^- = k,
	H_1^- = 0) \, \frac{\bbP(S_{n-k}=0)}{\bbP(S_n = 0)}
	\sim \sum_{k=1}^{\sqrt{n}} \bbP( \tau_1^- = k,
	H_1^- = 0) \to \bbP( H_1^- = 0) \,.
\end{equation*}
Since $\bbP(S_{n-k}=0)/\bbP(S_n = 0) \le (const.)$ for all $n \in \N$
and $\sqrt{n} < k \le \frac{n}{2}$, again by Gnedenko's local limit theorem \eqref{eq:Gnedenko},
it follows that as $n\to\infty$
\begin{equation*}
	\sum_{k=\sqrt{n}}^{n/2} \bbP( \tau_1^- = k,
	H_1^- = 0) \, \frac{\bbP(S_{n-k}=0)}{\bbP(S_n = 0)}
	\le (const.) \, \bbP(\tau_1^- > \sqrt{n}) \to 0 \,.
\end{equation*}
Finally, by \eqref{eq:abau} we have $\bbP( \tau_1^- = k,
H_1^- = 0) \le \frac{1}{k} \, \bbP(S_k = 0) \le (const.) k^{-1-1/\alpha}$.
Recalling that $\bbP(S_m = 0) \sim g(0) m^{-1/\alpha}$,
it follows that
\begin{equation*}
	\sum_{k=n/2}^n \bbP( \tau_1^- = k,
	H_1^- = 0) \, \frac{\bbP(S_{n-k}=0)}{\bbP(S_n = 0)}
	\le (const.') \left( \frac{n}{2} \right)^{-1-1/\alpha} \, n^{1/\alpha}
	\sum_{m=1}^{n/2} \frac{1}{m^{1/\alpha}} \to 0 \,,
\end{equation*}
as $n\to\infty$. These relations show that
$\bbP( H_1^- = 0 | S_n = 0) \to \bbP( H_1^- = 0)$ as $n\to\infty$, hence
$\bbP( H_1^- > 0 | S_n = 0) \to \bbP( H_1^- > 0)$ as $n\to\infty$.
Recalling \eqref{eq:abau} and Gnedenko's local limit theorem
\eqref{eq:Gnedenko}, as $n\to\infty$ we finally get
\begin{equation*}
	\widehat q_n^+(0,0) \sim \frac 1n\, \bbP(H_1^- > 0) \, \bbP(S_n=0)
	\sim \frac 1n \, \bbP(H_1^->0) \, \frac{g(0)}{a_n} \,,
\end{equation*}
which coincides with the second relation
in \eqref{eq:xqsmall}, because $\underline {\widehat V}^+(0) = \bbP(H_1^\pm > 0) = 1-\zeta$.


\subsection{Proof of the first relation in \eqref{eq:xqsmall} in the case $x=y=0$}
\label{sec:45-1}

We set $K(n) := \widehat q_n^+(0,0) = \bbP
(\tau_1^- = n, H_1^- = 0)$
and note that $\sum_{n\in\N} K(n) = \bbP(H_1^- = 0) = \zeta \in (0,1)$,
so that $K(\cdot)$ may be viewed as a defective probability on $\N$.
Summing over the location of the times $t \le n$ at which $S_t = 0$, we obtain
for all $n\in\N$
\begin{equation} \label{eq:qn+00}
	q_n^+(0,0) = \sum_{m=1}^\infty \sum_{0 =: t_0 < t_1 < \ldots < t_m = n}
	K(t_1) K(t_2 - t_1) \cdots K(t_m - t_{m-1}) =
	\sum_{m = 0}^\infty K^{*m}(n) \,,
\end{equation}
where $K^{*m}(\cdot)$ denotes the $k$-fold convolution of $K(\cdot)$ with itself.
This shows that $q_n^+(0,0)$ may be viewed as the renewal mass function associated
to the renewal process with inter-arrival
\emph{defective} probability $K(\cdot)$. Since $K(n) \in R_{-1-1/\alpha}$
by the second relation in \eqref{eq:xqsmall}, the asymptotic behavior
of $q_n^+(0,0)$ as $n\to\infty$ is a classical result in heavy-tailed
renewal theory (cf. \cite[Theorem~A.4]{cf:Gia} or
\cite[Proposition~12]{cf:AFK} for a more general result):
\begin{equation*}
	q_n^+(0,0) \sim \frac{1}{(1- \sum_{n\in\N}K(n))^{2}} \, K(n)
	= \frac{1}{(1-\zeta)^{2}} \,\underline {\widehat V}^+(0)\,
	\frac{g(0)}{n\, a_n} \,.
\end{equation*}
By definition $\underline {\widehat V}^+(0) = 1 - \zeta$, cf. \eqref{eq:VV+0}
while $V^+(0) = (1-\zeta)^{-1}$, cf. the line following \eqref{eq:defV},
therefore the first relation in \eqref{eq:xqsmall} is proven for
case $x=y=0$.


\subsection{Proof of relation \eqref{eq:usefuly0}}
\label{sec:319}

Let us rewrite \eqref{eq:usefuly0} for convenience:
as $n\to\infty$
\begin{equation} \label{eq:usefuly}
	\bbP(\widehat \tau_1^- > n) \sim (1-\zeta)^{-1} \bbP( \tau_1^- > n) \,.
\end{equation}
Summing on the position of the last epoch before $n$ at which
the random walk hits zero, we can write
\begin{equation*}
	\bbP(\widehat \tau_1^->n)
	= \sum_{\ell=0}^n \bbP( S_i \ge 0\ \forall 0 \le i \le \ell\,,
	\ S_\ell = 0)
	\,\bbP( \tau_1^- > n-\ell) =
	\sum_{\ell=0}^n q_\ell^+(0,0)
	\,\bbP( \tau_1^- > n-\ell)\,.
\end{equation*}
From \eqref{eq:qn+00} and the preceding lines, we have $\sum_{\ell= 0}^\infty
q_\ell^+(0,0) = \sum_{k=0}^\infty \zeta^k = (1-\zeta)^{-1}$.
Since $\bbP( \tau_1^- > n-\ell) \sim \bbP( \tau_1^- > n)$ as $n\to\infty$
uniformly for $\ell \le \log n$ (we recall that $\bbP( \tau_1^- > n)$
is regularly varying), it follows that
\begin{equation*}
	\sum_{\ell=0}^{\lfloor \log n \rfloor} q_\ell^+(0,0)
	\,\bbP( \tau_1^- > n-\ell) \sim (1-\zeta)^{-1} \bbP( \tau_1^- > n) \,,
\end{equation*}
in agreement with \eqref{eq:usefuly}. It remains to show that the terms
with $\ell \ge \log n$ give a negligible contribution.
Since $\bbP( \tau_1^- > n-\ell) \le (const.) \bbP( \tau_1^- > n)$
for $\log n \le \ell \le n/2$, we can write
\begin{equation*}
	\sum_{\ell=\lfloor \log n \rfloor}^{\lfloor n/2 \rfloor} q_\ell^+(0,0)
	\,\bbP( \tau_1^- > n-\ell) \le (const.)
	\bigg(\sum_{\ell \ge \log n} q_\ell^+(0,0) \bigg)
	\bbP( \tau_1^- > n) = o\big(\bbP( \tau_1^- > n)\big) \,.
\end{equation*}
Finally, for $\ell \ge n/2$ we have
$q_\ell^+(0,0) \le (const.)/(n a_n)$ by \eqref{eq:xqsmall}.
Furthermore, we also have
$\sum_{\ell = \lfloor n/2 \rfloor}^{n} \bbP(\tau_1^- > n-\ell)
\sim \rho^{-1} n \bbP(\tau_1^- > n/2)$,
because $\bbP(\tau_1^- > m)$ is regularly varying with
index $-(1-\rho) > -1$, cf. \eqref{eq:sumup2}. It follows that
\begin{equation*}
	\sum_{\ell=\lfloor n/2 \rfloor}^n q_\ell^+(0,0)
	\,\bbP( \tau_1^- > n-\ell) \le (const.) \frac{1}{n a_n} \,
	n \, \bbP( \tau_1^- > n) = o\big(\bbP( \tau_1^- > n)\big) \,,
\end{equation*}
and the proof of \eqref{eq:usefuly} is complete.

\section{Proof of Lemma~\ref{th:density}}
\label{ap:B}

The continuity of the function $x\mapsto f_{\gep, t}(x)$ on $[0,\infty)$ is a consequence
of its definition:
\begin{equation}\label{5651}
f_{\gep, t}(x)=\frac{g_\varepsilon^\uparrow(x,0)}{g_t^\uparrow(0,0)} \,,
\end{equation}
thanks to the continuity of $x\mapsto g_\varepsilon^\uparrow(x,0)$,
which is proved in Lemma 3 of \cite{cf:Ur11}.

We recall that $\lambda^\uparrow(\dd z) = z^\alpha \, \dd z$ and
\begin{equation*}
	g^\uparrow_t(x, y) \,:=\, \frac{\bP^\uparrow_x(X_t \in \dd y)}{\lambda^\uparrow(\dd y)}\,.
\end{equation*}
Since $\{c^{-1/\alpha} X_{ct}\}_{t\ge 0}$ has the same law as $\{X_t\}_{t\ge 0}$,
it follows that
\begin{equation} \label{eq:scalingg}
	g^\uparrow_t(x,y) \,=\, \frac{1}{t^{1+1/\alpha}} \, g^\uparrow_1\bigg(\frac{x}{t^{1/\alpha}},
	\frac{y}{t^{1/\alpha}} \bigg) \,, \qquad \forall t > 0, \
	\forall x,y \in [0,\infty) \,.
\end{equation}
In particular
\begin{equation}\label{eq:garr}
	g^\uparrow_t(0,0) = t^{-1-1/\alpha} g^\uparrow_1(0,0)
\end{equation}
and equation \eqref{eq:asfgep} follows:
\begin{equation} \label{eq:ft0}
	f_{\epsilon,t}(0) \,=\, \frac{g^\uparrow_\epsilon(0,0)}{g^\uparrow_t(0,0)}
	\,=\, \bigg( \frac{t}{\epsilon} \bigg)^{1+1/\alpha} \,.
\end{equation}

It only remains to prove \eqref{eq:density}.
Let us consider the reflected L\'evy process $\widetilde X := -X$.
By part~1 of Theorem~1 in \cite{cf:Cha97},
the law of the terminal value of the meander of this process,
that is $g^-(x) \, \dd x$, is absolutely continuous
with respect to the law of $\widetilde X_1$ under $\bP^\uparrow$,
that is $\widetilde g^\uparrow_1(0,x) \, \lambda^\uparrow(\dd x)$
(with obvious notations), with Radon-Nikodym density proportional to $1/\widetilde U^-(x)
= x^{-\alpha\rho}$.
Therefore there is a constant $C > 0$ such that for all $x \ge 0$
\begin{equation*}
	g^-(x) \,=\, C \, \widetilde g^\uparrow_1(0,x) \, x^\alpha \, \frac{1}{x^{\alpha\rho}} \,,
	\qquad \forall x \ge 0 \,.
\end{equation*}
Now observe that $g^\uparrow_t(a,b) = g^\uparrow_t(b,a)$ for all $a,b \ge 0$,
as it is proved in Lemma~1 of~\cite{cf:Ur11}.
(More directly, it is enough to check this relation for $a,b > 0$, by continuity,
and this is evident from \eqref{eq:g+}.)
Therefore we can write
\begin{equation*}
	g^\uparrow_1(x,0) \,=\, \frac{1}{C} \, \frac{g^-(x)}
	{x^{\alpha(1-\rho)}} \,,
\end{equation*}
and by the scaling relation \eqref{eq:scalingg} it follows that
\begin{equation*}
	{g}_\gep^\uparrow(x,0)=\frac{\gep^{-\rho-1/\alpha}}
	{C}\frac{g^-(\epsilon^{-1/\alpha} x)}{x^{\alpha(1-\rho)}} \,.
\end{equation*}
Looking back at \eqref{5651} and \eqref{eq:garr} we then obtain
\begin{equation} \label{eq:asi}
	f_{\epsilon,t}(x) \,=\, \frac{\gep^{-\rho-1/\alpha}}
	{C \, g^\uparrow_1(0,0) \,
	t^{1-1/\alpha}}\frac{g^-(\epsilon^{-1/\alpha} x)}{x^{\alpha(1-\rho)}}
	\,=\, \frac{(t/\gep)^{1+1/\alpha}}
	{C \, g^\uparrow_1(0,0)}\frac{g^-(\epsilon^{-1/\alpha} x)}
	{(\epsilon^{-1/\alpha}x)^{\alpha(1-\rho)}} \,.
\end{equation}
Recalling \eqref{eq:DonSav} and \eqref{eq:C+-}, it follows from this expression that
\begin{equation*}
	\lim_{x \downarrow 0} f_{\epsilon, t}(x) \,=\,
	\frac{\mathtt{C}^- g(0)}{C \, g^\uparrow_1(0,0)} \, (t/\gep)^{1+1/\alpha} \,,
\end{equation*}
and looking back at \eqref{eq:ft0} we see that
$C \, g^\uparrow_1(0,0) = \mathtt{C}^- g(0)$.
But then \eqref{eq:asi} coincides with \eqref{eq:density},
and the proof is completed.


%


\end{document}